\documentclass[11pt]{article}

\usepackage[letterpaper,margin=1in]{geometry}
\usepackage{tikz,pgfplots}
\usepackage{epstopdf}% To incorporate .eps illustrations using PDFLaTeX, etc.
\usepackage{subcaption}
\usepackage{amsfonts,amsmath,amssymb}
\usepackage{amsthm}
\usepackage[normalem]{ulem}
\usepackage{booktabs}
\usepackage{siunitx}
\usepackage[ruled, vlined]{algorithm2e}
\DontPrintSemicolon
\SetInd{0.3em}{0.5em}

\usepackage{xcolor}
\usepackage{bm}
\usepackage[colorlinks=false, urlbordercolor=red]{hyperref}
\usepackage{textpos}
\usepackage{lscape}
\usepackage{graphicx}
\usepackage{multirow}
\usepackage{tikz,pgfplots}
\usetikzlibrary{arrows,shapes,positioning,fit,calc,backgrounds}
\usetikzlibrary{patterns}
\usetikzlibrary{quotes,angles}
\tikzset{
  reduce height/.style={
    minimum height=0pt,
    inner ysep=#1,
    text depth=2pt
  },
  reduce height/.default={0pt}
}
\usepackage{textcomp}
\usepackage{multicol}
\usepackage{orcidlink}
\usepackage{float}
\usepackage{pifont}
\usepackage{authblk}

\pgfplotsset{compat=1.15}
\usepackage{geometry}
\renewcommand{\arraystretch}{1.2}

\usepackage{colortbl}
\usepackage{pdflscape}
\usepackage{cite}               % ... more powerful citations package
\usepackage{courier}
\usepackage{listings}
\usepackage[toc]{appendix}
\usepackage{longtable}
\usepackage{environ}
\usepackage{etoolbox}

\newcommand\equaldef{\stackrel{\text{def}}{=}}

\usepackage{inconsolata}
\usetikzlibrary{arrows.meta}

\newtheorem{definition}{Definition}
\newtheorem{lemma}{Lemma}

\newtheorem{theorem}{Theorem}
\newtheorem{assumption}{Assumption}
\newtheorem{remark}{Remark}

\definecolor{constraint_relaxation_color}{HTML}{cab2d6}
\definecolor{inequality_method_color}{HTML}{92b9e1}
\definecolor{mechanism_color}{HTML}{FFCCAC}
\definecolor{strategy_color}{HTML}{FFEB94}
\definecolor{subproblem_color}{HTML}{C1E1DC}

\NewEnviron{globalization_mechanism}{%
	\colorbox{mechanism_color}{
		\parbox{0.97\hsize}{
			\begin{textblock*}{95pt}[1, 1](\hsize,6pt)
			\emph{\scriptsize{globalization mechanism}}
			\end{textblock*}%
			\BODY
		}
	}
}

\NewEnviron{globalization_mechanism_no_text}{%
	\colorbox{mechanism_color}{
		\parbox{0.97\hsize}{
			\begin{textblock*}{95pt}[1, 1](\hsize,6pt)
			\end{textblock*}%
			\BODY
		}
	}
}

\NewEnviron{globalization_strategy}{%
	\colorbox{strategy_color}{
		\parbox{0.97\hsize}{
			\begin{textblock*}{85pt}[1, 1](\hsize,5pt)
			\emph{\scriptsize{globalization strategy}}
			\end{textblock*}%
			\BODY
		}
	}
}
\NewEnviron{constraint_relaxation}{%
	\colorbox{constraint_relaxation_color}{
		\parbox{0.97\hsize}{
			\begin{textblock*}{85pt}[1, 1](\hsize,5pt)
			\emph{\scriptsize{constraint relaxation}}
			\end{textblock*}%
			\BODY
		}
	}
}
\NewEnviron{subproblem}{%
	\colorbox{subproblem_color}{
		\parbox{0.97\hsize}{
			\begin{textblock*}{40pt}[1, 1](\hsize,6pt)
			\emph{\scriptsize{subproblem}}
			\end{textblock*}%
			\BODY
		}
	}
}

\newcommand{\highlight}[2][yellow]{\mathchoice%
  {\colorbox{#1}{$\displaystyle#2$}}%
  {\colorbox{#1}{$\textstyle#2$}}%
  {\colorbox{#1}{$\scriptstyle#2$}}%
  {\colorbox{#1}{$\scriptscriptstyle#2$}}}%
  
\newcommand{\R}{\mathbb{R}}
\newcommand{\N}{\mathbb{N}}

\newcommand{\mini}{\mathop{\mbox{min}}}

\newcommand{\st}{\mbox{s.t.}}
\newcommand{\dps}{\displaystyle}

\newcommand{\constraintmultipliers}{\lambda}

\newcommand{\boundmultipliers}{\mu}

\newcommand{\Lag}{{\cal L}}

\newcommand{\trial}[1]{\hat{#1}}

\newcommand{\TRradius}{\Delta_{\text{TR}}}
\newcommand{\papertitle}{A Unified Funnel Restoration SQP Algorithm}

\title{\line(1,0){470}\\\textbf{\papertitle}\footnote{This work has been carried out within the framework of the Flanders Make SBO project DIRAC: Deterministic and Inexpensive Realizations of Advanced Control. This work was also supported by the Applied Mathematics activity within the U.S. Department of Energy, Office of Science, Advanced Scientific Computing Research, under contract number DE-AC02-06CH11357.}\\\line(1,0){470}}

\author[1]{David Kiessling}
\author[2]{Sven Leyffer}
\author[2,3]{Charlie Vanaret}
\affil[1]{Department of Mechanical Engineering, KU Leuven and Flanders Make @ KU Leuven, Leuven, Belgium}
\affil[2]{Mathematics and Computer Science Division, Argonne National Laboratory, Lemont, IL, USA}
\affil[3]{Mathematical Algorithmic Intelligence Division, Zuse-Institut Berlin, Germany}

\date{}
\begin{document}

\maketitle

\begin{abstract}
%Outline:
%\begin{itemize}
%\item Nonlinearly constrained optimization: importance.
%\item Unified framework allows us to express a variaty of solvers.
%\item Framework implications for software design; simplify implementation of new algorithms.
%\item Example: new look at funnel; methods.
%\item Convergence analysis of funnel methods and similarity to filter methods motivated by framework.
%\item Numerical results.
%\end{itemize}
We consider nonlinearly constrained optimization problems and discuss a generic double-loop framework consisting of four algorithmic ingredients that unifies a broad range of nonlinear optimization solvers. This framework has been implemented in the open-source solver \texttt{Uno}, a Swiss Army knife-like C++ optimization framework that unifies many nonlinearly constrained nonconvex optimization solvers.
We illustrate the framework with a sequential quadratic programming (SQP) algorithm that maintains an acceptable upper bound on the constraint violation, called a funnel, that is monotonically decreased to control the feasibility of the iterates. Infeasible quadratic subproblems are handled by a feasibility restoration strategy. Globalization is controlled by a line search or a trust-region method.
We prove global convergence of the trust-region funnel SQP method, building on known results from filter methods. We implement the algorithm in \texttt{Uno}, and we provide extensive test results for the trust-region line-search funnel SQP on small \texttt{CUTEst} instances. 
%\sven{Update abstract at end!}
\end{abstract}

\section{Introduction and Background}

% \paragraph{Outline/Content.}
% \begin{itemize}
%     \item Nonlinearly constrained optimization (NCO?)
%     \item Applications \& need for reliable solvers (MIP/MPEC)
%     \item methods: iterative (Newton \& FO); structure?
%     \item review of methods \& solvers
%     \item {\bf Motivation for Uno:} fast, reliable, and expandable solvers; generalizable to MPEC, robust, and other emerging applications
% \end{itemize}

We focus on algorithms for solving  nonlinearly constrained optimization problems (NCOs) of the form
\begin{equation}
\label{eq:NCO}
\tag{NCO}
\begin{array}{ll} \dps
\mini_{x\in\mathbb{R}^n}  f(x),\quad\st\quad c(x) = 0,\quad x \ge 0,
\end{array}
\end{equation}
where $f \colon \R^n \to \R$ and $c \colon \R^n \to \R^m$ are twice continuously differentiable and possibly nonconvex. Problems with general inequalities $l \le c(x) \le u$ with $l, u \in \R^m$ can be formulated as \ref{eq:NCO} via the introduction of slack variables.
% Our algorithms use a funnel method as a globalization strategy, that is a way to assess whether a trial iterate makes sufficient progress towards a solution.
% % Other ingredients necessary for global convergence are varied.
% Special focus is put on funnel methods from the perspective of filter methods, which provides the basis for a unifying global convergence approach.

NCOs arise in many important applications, such as optimal control problems \cite{betts2010practical}, partial differential equation constrained optimization problems \cite{antil2018frontiers,herzog2010algorithms} that model topology optimization \cite{bendsoe2013topology}, inverse problems \cite{biegler2007real}, or control problems \cite{troltzsch2010optimal}. In addition, NCOs arise as subproblems in more complex optimization problems such as mixed-integer nonlinear optimization \cite{lee2011mixed} and optimization problems with complementarity constraints \cite{outrata1998nonsmooth}.

Solution methods for NCOs are iterative and generate a sequence of iterates $x^{(k)}$ for $k \ge 0$ that (hopefully) converges to a stationary point of NCO or a stationary point of the constraint violation $\| c(x)\|$ under mild assumptions when started near a stationary point. In general, however, we must safeguard our methods to ensure convergence from remote starting points (which we refer to as ``global convergence" in the remainder). Iterative methods for NCOs share common algorithmic features, such as how new iterates are computed and how global convergence is ensured.

Here we present a generic framework for solving NCOs as a double-loop algorithm. The outer loop computes new iterates that converge to a stationary point, while the inner loop implements the convergence safeguards. This perspective allows us to easily interpret many existing NCO methods within this framework and {\em forms the basis of an open-source C++ implementation that unifies many existing solvers for NCOs, called \texttt{Uno}}~\cite{vanaret2024,VanaretLeyffer2024}.
%, or unified nonlinear optimizer.}
\texttt{Uno} is a reliable and modular solver for NCO that is meant to be extended to 
% Our vision for Uno is to enable us to expand NCO techniques to
other areas such as optimization with equilibrium constraints (see, e.g., \cite{outrata1998nonsmooth,fletcher2004solving,fletcher2006local}) or robust optimization \cite{leyffer2020survey}.

In this paper we concentrate on \texttt{funnel methods} to promote global convergence. We relate the funnel idea to filter methods \cite{fletcher1997,fletcher2002a} and illustrate how it fits into the double-loop framework. This allows us to quickly develop a funnel implementation by modifying the existing filter implementation within \texttt{Uno}.
The  funnel method has been studied in \cite{gould2010,gould2011} for solving equality-constrained NCOs with a Byrd--Omojokun trust-region method.
%An error in a proof was later corrected \cite{gould2011}.
An iteration consists of the solution of several subproblems that yield tangential and normal step decomposition. Iterations are divided into three different kinds, namely, $f$-type, $v$-type, and $y$-type iterations, which account for an improvement in optimality, a reduction of infeasibility, or an update of the Lagrange multipliers. In his Ph.D. thesis \cite{Samadi2018}, Samedi used a funnel in the context of equality-constrained optimization to design an algorithm with favorable worst-case complexity bounds. The first funnel method for solving inequality-constrained NCOs was introduced in \cite{curtis2017}: the problem was transformed into a barrier problem and solved with a Lagrange--Newton method. The approach uses matrix-free and inexact methods and performs well on large-scale NCOs. Funnel methods can be traced back to the Ph.D. thesis of Zoppke-Donaldson \cite{Zoppke1995} that implements an SQP method that employs a tolerance tube and reports encouraging numerical results but without any convergence analysis. 

We look at funnel methods through the lens of filter methods, from which we derive our global convergence proof.
We wish to demonstrate that the funnel method obtains  performance similar to that of the filter method, while being simpler to implement.
% Charlie: I feel this is redundant and already discussed in Section 2
% Filter methods handle the NCO in the spirit of a bi-objective optimization problem by considering optimality and feasibility separately. They are an alternative to merit functions for globalizing NCOs, without the need to update a penalty parameter.
Filter methods were introduced in \cite{fletcher1997} in the context of a trust-region SQP algorithm that switches to a feasibility restoration phase when the quadratic subproblem is infeasible. The method was implemented in the solver \texttt{filterSQP}, which obtained excellent performance on the \texttt{CUTEst} \cite{gould2014} test set. A filter line-search interior-point method was implemented within the solver \texttt{IPOPT}, one of the most successful open-source NCO solvers \cite{waechter2006}. Respective global convergence proofs were given in \cite{fletcher2002a}, \cite{fletcher2002b}, and \cite{waechter2005}. The latter two papers serve as a baseline for proving global convergence results.
 
% SL
\paragraph{Contributions.}
%\todo{The contributions should appear somewhere.}
We make a number of important contributions in this paper: 
(i) we investigate the funnel method from the perspective of filter methods; 
(ii) we consider equality and inequality-constrained problems (if necessary by introducing additional slack variables); 
(iii) we prove global convergence of a trust-region restoration funnel method; and
(iv) we implement a trust-region restoration funnel method and a line-search restoration funnel method in the \texttt{Uno} solver, and we provide extensive numerical results on a subset of the \texttt{CUTEst} test set.
To our knowledge, this is the first time that line-search funnel methods have been considered.

\paragraph{Outline.} This paper is structured as follows. In the remainder of this section we introduce our notation and necessary optimality conditions.
Next, we present the double-loop framework and the main algorithmic components that make up a generic NCO method, using the funnel method as an exemplar.
%\todo{is examplar the right word? Yes.}. 
We then review the funnel method in more detail and prove global convergence. We present numerical results on a subset of the \texttt{CUTEst} test set (with detailed tables of function and derivative evaluations for the trust-region and line-search methods in the appendices). In the final section we summarize our conclusions.  
% In Section \ref{sec:problem_statement}, the general problem formulation and important notions of nonlinear optimization are introduced. Section \ref{sec:unified_abstraction_nl_opti} defines the abstract framework for optimization algorithms. Section \ref{sec:funnel_mechanism} introduces the funnel method and discuss its relation to filter methods. The funnel restoration SQP framework is introduced in Section \ref{sec:funnel_restoration_sqp}. Global convergence proofs are provided in Section \ref{sec:global_convergence_proofs}. Simulation results are provided in Section \ref{sec:simulation_results}. Section \ref{sec:conclusion} concludes this paper.

\subsection*{Notation and Necessary Optimality Conditions}
\label{sec:problem_statement}

The $j$th component of a vector $x\in\R^n$ is denoted by a subscript,  $x_j$. The $k$th iteration index is given by a superscript,  $x^{(k)}$. For brevity, quantities evaluated at a given iterate $x^{(k)}$ are denoted by a superscript as well, for example, $f^{(k)} \equaldef f(x^{(k)})$. 

We start by defining the Fritz John function (or scaled Lagrangian) of \eqref{eq:NCO} at $(x, \rho, \constraintmultipliers, \boundmultipliers)$:
\begin{equation}
\Lag(x, \rho, \constraintmultipliers, \boundmultipliers) \equaldef \rho f(x) - \constraintmultipliers^T c(x)  - \boundmultipliers^T x = \rho f(x) - \sum_{j=1}^m \constraintmultipliers_j c_j(x) - \sum_{i=1}^n \boundmultipliers_i x_i,
\end{equation}
where $\constraintmultipliers \in \R^m$ and $\boundmultipliers \in \R^n$ are the Lagrange multipliers of the equality constraints and the bound constraints, respectively, and $\rho \in \{0, 1\}$.
We use the scaled Lagrangian because it allows us to treat feasible and infeasible stationary points of \eqref{eq:NCO} in a unified way.
When $\rho = 1$, we obtain the standard Lagrangian.
The gradient of the scaled Lagrangian with respect to $x$ at a point $(x, \rho,  \constraintmultipliers, \boundmultipliers)$ is denoted by
\begin{equation}\label{eq:statKKT}
\nabla_x \Lag(x, \rho, \constraintmultipliers, \boundmultipliers) \equaldef \rho\nabla f(x) - \sum\limits_{j = 1}^m \constraintmultipliers_j \nabla c_j(x) - \boundmultipliers,
\end{equation}
where $\nabla f(x) \in \R^n$ is the gradient of $f$ and $\nabla c(x)^{T}\in\R^{m \times n}$ is the Jacobian of $c$. 
% (the latter correspond to infeasible points that satisfy \eqref{eq:statKKT} with $\rho = 0$.

The Hessian of the scaled Lagrangian with respect to $x$ at $(x, \rho, \constraintmultipliers)$ is defined by
\begin{equation}
\label{eq:lagrangian-hessian}
W_\rho(x, \constraintmultipliers) \equaldef \nabla^2_{xx} \Lag(x, \rho, \constraintmultipliers, \boundmultipliers) = \rho\nabla^2 f(x) - \sum_{j=1}^m \constraintmultipliers_j \nabla^2 c_j(x),
\end{equation}
where $\nabla^2 f(x) \in \R^{n \times n}$ is the Hessian of $f$ and $\nabla^2 c_j(x) \in \R^{n \times n}$ is the Hessian of $c_j$.

We say that a feasible point $x$ of \eqref{eq:NCO} fulfills the Mangasarian--Fromowitz constraint qualification (MFCQ) if the gradients of the constraints $c(x)$ are linearly independent at $x$ and if there exists $d \in \mathbb{R}^n$ such that $\nabla c(x)^T d = 0$ and $d_i > 0$ if $x_i = 0$ ($d$ points into the interior of the feasible set). Necessary optimality conditions for \eqref{eq:NCO} are given in the following definition.

\begin{definition}[KKT conditions~\cite{Nocedal2006}]
If MFCQ holds at an optimal point $x^*$ of \eqref{eq:NCO}, the first-order necessary optimality conditions of \eqref{eq:NCO} at $x^*$ state that there exist multipliers $\constraintmultipliers^* \in \R^m$ and $\boundmultipliers^* \in \R^n$ such that
% Charlie: I made it a bit more compact. Feel free to revert
\begin{align*}
\nabla_x \Lag(x^*, 1, \constraintmultipliers^*, \boundmultipliers^*) = 0, \quad
c(x^*) = 0, \quad
x^* \geq 0, \quad
\boundmultipliers^* \geq 0, \quad
x^* \odot \boundmultipliers^* = 0,
\end{align*}
where $\odot$ denotes the Hadamard (componentwise) product.
These conditions are called the Karush--Kuhn--Tucker (KKT) conditions.    
\end{definition}
\section{Unified Abstraction of Nonlinearly Constrained Optimization}
\label{sec:unified_abstraction_nl_opti}

In \cite{VanaretLeyffer2024} we introduced an abstract framework to unify the workflows of iterative methods for NCO, arguing that most methods can be assembled by combining the following four generic ingredients within a double-loop framework:
\begin{enumerate}
\item A \colorbox{constraint_relaxation_color}{constraint relaxation strategy} is a systematic way to reformulate \eqref{eq:NCO} with relaxed constraints,for example, feasibility restoration or $\ell_1$ relaxation.
\item A \colorbox{subproblem_color}{subproblem method} is a local approximation of the reformulated NCO at the current primal-dual iterate, for example, inequality-constrained quadratic problems (QPs) in an SQP framework or a primal-dual interior-point Newton system.
\item A \colorbox{strategy_color}{globalization strategy} assesses whether a trial iterate makes sufficient progress toward a solution, for example, filter method or $\ell_1$ merit function.
\item A \colorbox{mechanism_color}{globalization mechanism} defines the action that an algorithm takes when a trial iterate is not acceptable, for example, line-search or trust-region method.
\end{enumerate}

The double-loop framework portrayed in Algorithm~\ref{alg:frameworkUNO} shows how the four ingredients interact with one another. This abstract framework is implemented in \texttt{Uno}; it offers robust, off-the-shelf strategies that are independent and agnostic of each other. Strategy combinations can be assembled on the fly with no programming effort from the user. In particular, \texttt{Uno} implements three presets that mimic existing solvers: a filterSQP~\cite{fletcher1997} preset (a trust-region restoration filter SQP method); an IPOPT~\cite{waechter2005,waechter2006} preset (a line-search restoration filter barrier method); and a Byrd~\cite{byrd2008steering} preset (a line-search $\ell_1$-merit S$\ell_1$QP method).

\begin{algorithm}
\SetAlgoVlined
\caption{Abstract double-loop framework for iterative methods for \eqref{eq:NCO}.}
\label{alg:frameworkUNO}
\small
\KwData{initial point $(x^{(0)}, \constraintmultipliers^{(0)}, \boundmultipliers^{(0)})$}
Set $k \gets 0$ \;
\While{termination criteria at $(x^{(k)}, \constraintmultipliers^{(k)}, \boundmultipliers^{(k)})$ not met}{
    \begin{globalization_mechanism}
    \Repeat{$(\trial{x}^{(k+1)}, \trial{\constraintmultipliers}^{(k+1)}, \trial{\boundmultipliers}^{(k+1)})$ is \colorbox{strategy_color}{acceptable}}{
        Solve \colorbox{constraint_relaxation_color}{feasible} \colorbox{subproblem_color}{subproblem}(s) that approximate(s) \eqref{eq:NCO} at $(x^{(k)}, \constraintmultipliers^{(k)}, \boundmultipliers^{(k)})$ \;
        Assemble trial iterate $(\trial{x}^{(k+1)}, \trial{\constraintmultipliers}^{(k+1)}, \trial{\boundmultipliers}^{(k+1)})$ \;
    }
    \end{globalization_mechanism}
    Update $(x^{(k+1)}, \constraintmultipliers^{(k+1)}, \boundmultipliers^{(k+1)}) \gets (\trial{x}^{(k+1)}, \trial{\constraintmultipliers}^{(k+1)}, \trial{\boundmultipliers}^{(k+1)})$ \;
    $k \gets k+1$ \;
} % end while OUTER
\KwResult{$(x^{(k)}, \constraintmultipliers^{(k)}, \boundmultipliers^{(k)})$}
\end{algorithm}

%The coloring above is used throughout this paper to illustrate how these four ingredients interact with each other in order to define an optimization algorithm. 
This paper uses the same abstraction as in \cite{VanaretLeyffer2024} but restricts the presentation to SQP methods solving {inequality-constrained QP} subproblems and {feasibility restoration} to ensure consistent QPs. The funnel method is investigated as a {globalization strategy} for different algorithmic configurations and unified both for trust-region and line-search methods. In the remainder of this section we provide details of these algorithmic ingredients.

\subsection{Subproblem Method: Sequential Quadratic Programming}
% Here, we chose a Sequential Quadratic Programming (SQP) method as our basic iterative solver to generate new iterates.
SQP is a second-order iterative method for finding a local solution for \eqref{eq:NCO}. At iteration $k$, it solves a local quadratic approximation of \eqref{eq:NCO} at the primal-dual iterate $(x^{(k)}, \lambda^{(k)}, \mu^{(k)})$:
\begin{equation}
\tag{QP$(x^{(k)})$}
\label{eq:QP-subproblem}
\begin{array}{lll} \dps
\mini_{d} & \frac{1}{2} d^T W_1^{(k)} d + (\nabla f^{(k)})^T d \\
\st 	& c^{(k)} + (\nabla c^{(k)})^T d = 0 \\
      & x^{(k)} + d \ge 0.
\end{array}
\end{equation}
The primal-dual solution of \eqref{eq:QP-subproblem} is denoted by $(d^{(k)}, \hat{\lambda}^{(k+1)}, \hat{\mu}^{(k+1)})$. The trial iterate for a given step size $\alpha \in (0, 1]$ is given by $\hat{x}^{(k+1)} = x^{(k)} + \alpha d^{(k)}$.
If the trial iterate is accepted by the globalization strategy, we move to the trial iterate and solve the next QP. % \todo{I think we talk about globalization strategy in the next paragraph}
% Charlie: they're mentioned at the beginning of Section 2
Otherwise, the globalization mechanism (e.g., a line search) defines a new trial iterate that is more likely to be acceptable.
The process terminates either when an approximate stationary (KKT) point is found or when a stationary point of the constraint violation is found or with an indication that a constraint qualification fails.

Under appropriate assumptions and close to a solution, SQP cnverges with $\alpha=1$ and achieves superlinear or quadratic local convergence \cite{Nocedal2006}.
Far from the solution, additional safeguards need to be taken to ensure 
%SL: done \todo{define global convergence} 
global convergence. SQP can suffer from ill-posedness; if the exact Hessian is indefinite, the QP can be unbounded. Moreover, the linearization of the constraints in \eqref{eq:NCO} can be inconsistent. Both cases yield iterations that are not well defined. Additionally, taking full steps ($\alpha = 1$) does not necessarily result in a converging method. A globalization mechanism needs to be incorporated, as well as a globalization strategy that ensures descent for the objective function and improvement in constraint violation. The next section discusses various methods to address these issues. 

\subsection{{Globalization Mechanism}: Trust Region or Line Search}

%\texttt{Uno} allows us to easily switch between two globalization mechanisms, namely
Standard globalization mechanisms include
trust-region methods and line-search methods. We discuss each scheme in turn and show that they can be interpreted as inner iterations. 

\subsubsection{Trust-Region Methods}

Trust-region methods limit the length of the direction $d$ a priori by imposing the trust-region constraint $\Vert d \Vert \leq \TRradius^{(l)}$, where $\TRradius^{(l)} > 0$ is the trust-region radius. Various norms are possible; in this paper we consider only the $\ell_\infty$ norm because it is most easily implemented within a QP subproblem. The step $d^{(k, l)}$ is obtained by solving the trust-region subproblem at the current primal-dual iterate $(x^{(k)}, \constraintmultipliers^{(k)}, \boundmultipliers^{(k)})$:
\begin{equation}
\label{eq:TR-QP}
\tag{$\mathit{QP}(x^{(k)},\TRradius^{(l)})$}
\begin{array}{lll} \dps
\min_{d} & \frac{1}{2} d^T W_1^{(k)} d + (\nabla f^{(k)})^T d \\
\st 	& c^{(k)} + (\nabla c^{(k)})^T d = 0 \\
		& x^{(k)} + d \ge 0 \\
		& \lVert d \rVert_\infty \le \TRradius^{(l)}.
\end{array}
\end{equation}
The radius $\TRradius^{(l)}$ is reduced until the trial iterate ${\trial{x}^{(k+1, l)} \equaldef x^{(k)} + d^{(k, l)}}$ is accepted by the globalization strategy or until \eqref{eq:TR-QP} becomes infeasible.
Usually, $\TRradius^{(l)}$ is increased in successful iterations if the trust region is active at $d^{(k, l)}$, and it is decreased to a value smaller than $\min(\TRradius^{(l)}, \lVert d^{(k,l)} \rVert_\infty)$ when the trial iterate is rejected. A positive definite Hessian $W_1^{(k)}$ is not required (as long as the QP solver can handle problems with an indefinite Hessian), because directions of negative curvature are bounded by the trust region.

\subsubsection{Line Search Methods}

Line-search methods solve \ref{eq:QP-subproblem} and search for an acceptable iterate along $d^{(k)}$ by varying the step size $\alpha^{(k,l)} \in (0, 1]$: the trial iterate is denoted by ${\trial{x}^{(k+1, l)} \equaldef x^{(k)} + \alpha^{(k,l)} d^{(k)}}$. Here we opt for a backtracking line search that seeks the largest step size in the sequence $a^{(k,l)}\in\{2^{-l}~|~l = 0, 1,\ldots\}$ such that the trial iterate is acceptable to the globalization strategy.
A positive definite approximation of the Hessian $W^{(k)}$ is required to guarantee the well-posedness of \ref{eq:QP-subproblem}.

The key difference with trust-region methods is that no additional QPs need be solved if a trial iterate is rejected. On the other hand, the step $\alpha^{(k,l)} d^{(k)}$ is usually not optimal (or even feasible) within the equivalent trust region, $\|d\|_{\infty} \le \alpha^{(k,l)} \| d^{(k)} \|_{\infty}$.
%\todo{This may need to be reformulated.}

\subsection{{Constraint Relaxation Strategy}: Feasibility Restoration}
\label{sec:feasibility-restoration}

An infeasible quadratic subproblem results from inconsistent linearized or bound constraints. This situation can indicate that \eqref{eq:NCO} is infeasible.
In this case the method reverts to the \textit{feasibility restoration phase}: the original objective is temporarily discarded, and the following feasibility problem (e.g., with the $\ell_1$ norm) is solved instead:
\begin{equation}
\begin{array}{ll} \dps
\mini_x & \| c(x) \|_1 \\
\st     & x \ge 0.
\end{array}
\label{eq:feasibility_problem}
\end{equation}
Other subproblems are also possible and do not influence the proof of global convergence.
%Other formulations are also possible. The feasibility problem \eqref{eq:feasibility_problem} can be reformulated as a smooth problem using nonnegative elastic variables $u, v \in \R^m$ that capture the positive and negative parts of $c(x)$, respectively. The smooth reformulation is given by
%\begin{equation}
%\begin{array}{lll} \dps
%\mini_{x,u,v} & e^T u + e^T v \\
%\st 	& c(x) - u + v = 0 \\
%		& x \ge 0, ~u \ge 0, ~ v \ge 0,
%\end{array}
%\label{eq:feasibility_problem_smooth}
%\end{equation}
Feasibility restoration improves feasibility until a minimum of the constraint violation is obtained or the subproblem becomes feasible again, in which case solving the original problem (the \textit{optimality phase}) is resumed.
Any (local) solution $x^* \geq 0$ of \eqref{eq:feasibility_problem} with $\| c(x^*) \|_1 > 0$ is an indication that \eqref{eq:NCO} is (locally) infeasible.
Because \eqref{eq:feasibility_problem} is essentially a bound-constrained problem, we will not analyze the convergence to infeasible stationary points here and instead assume without loss of generality that a convergent globalization is available.

In our implementation we use a smooth reformulation of the $\ell_1$ feasibility problem (possibly with a trust-region constraint). The feasible quadratic subproblem is given by
\begin{equation}
%\tag{$\mathit{FQP}^{(k)}(\TRradius)$}
\tag{$\mathit{FQP}(x^{(k)},\TRradius^{(l)})$}
\label{eq:FQP}
\begin{array}{lll} \dps
\mini_{d, u,v} & \frac{1}{2} d^T W_0^{(k)} d ~+~ e^T u + e^T v \\
\st 	& c^{(k)} + (\nabla c^{(k)})^T d ~- u + v = 0 \\
        & x^{(k)} + d \ge 0 , \quad \left( \| d \|_{\infty} \le \TRradius^{(l)}\right)\\
        %& \lVert d \rVert_\infty \le \TRradius \quad \text{ONLY IN TR METHODS} \\
        & u \ge 0, ~ v \ge 0,
\end{array}
\end{equation}
where we indicate the presence/absence of the trust-region constraint for trust-region/line-search methods, respectively.
Here $e$ is a vector of ones of appropriate size.
We note that introducing the elastic variables $u, v$ makes the constraint Jacobian full rank, which guarantees linear independence constraint qualification. Moreover, we  add the trust-region bound only to the original variables, thus ensuring that \eqref{eq:FQP} is feasible for any trust-region radius. For the line-search variant, we will refer to $(FQP(x^{(k)}))$.

\subsection{Globalization Strategies}

Constrained optimization is concerned with the realization of two competing goals: minimizing the constraint violation and minimizing the objective value.
Filter methods~\cite{fletcher1997,fletcher2002a} interpret \eqref{eq:NCO} as a bi-objective optimization problem. Instead of combining the objective and the constraint violation in a merit function, they decompose the optimization problem into two separate goals: reducing the objective function and reducing the constraint violation (the latter takes precedence).
The great advantage is that they do not require a priori knowledge or the update of a penalty parameter.
Funnel methods adaptively define an acceptable threshold of constraint violation and accept steps that satisfy a sufficient decrease condition either for the constraint violation or for the objective. The rationale for funnel methods is that close to a local minimum and near the feasible set, we can expect the QP model to be a good predictor for the decrease of the objective, while far from the feasible set the QP step is more likely to reduce the constraint violation. Both the funnel and filter methods employ a natural switching condition that is capable of recognizing these two scenarios.
In the following, we first discuss the common progress measures for the funnel and filter methods, then introduce the methods in detail, and highlight their close connections and similarities.

%\todo[inline]{
%Add small insightful example that shows similarity and difference to filter methods? \\
%Example filter Maratos of brief history of filter methods. If we do full step with default settings.
%Put it in the paragraph here.
%}

\subsubsection{Progress Measures}
%\todo[inline]{We need this section because these are definitions that hold for all the globalization strategies}
% \textbf{GAIL - this section does, however, seem odd} since you just said you were going to detail the two methods and instead you do this. Perhaps add a lead-in sentence as the penultimate sentence above -- something like In the following, we first present some definitions for our progress metrics. We then detail...
% Charlie: fixed

Without loss of generality, the bound constraints on $x$ are always feasible throughout SQP iterations.
To quantify progress regarding the constraint violation or the objective, we monitor the measure of infeasibility $h(x) \equaldef \Vert c(x) \Vert_1$ and the objective $f(x)$ throughout the optimization process.
Local models of $h(x)$ and $f(x)$ at an iterate $x^{(k)}$ are defined by
\begin{subequations}
\label{eq:local_models}
\begin{align}
&m_h^{(k)}(d) \equaldef \big\lVert c^{(k)} + (\nabla c^{(k)})^T d \big\rVert_1, \\
&m_f^{(k)}(d) \equaldef \frac{1}{2}d^{T} W_1^{(k)} d + (\nabla f^{(k)})^T d + f^{(k)}.
\end{align}
\end{subequations}
We define the respective \textit{predicted reductions} by
\begin{subequations}
\label{eq:pred}
\begin{align}
&\Delta m_h^{(k)}(d) \equaldef m_h^{(k)}(0) - m_h^{(k)}(d) = h^{(k)} - m_h^{(k)}(d),\\
&\Delta m_f^{(k)}(d) \equaldef m_f^{(k)}(0) - m_f^{(k)}(d) = -\frac{1}{2}d^{T} W_1^{(k)} d - (\nabla f^{(k)})^T d.
\end{align}
\end{subequations}
% Note that for an unconstrained optimization problem and for a feasible point the predicted reduction is always greater or equal $0$ (also for a nonconvex QP solver, if the solver follows the negative curvature like BQPD). That indicates that in the restoration phase, the predicted reduction should always be positive.
The \textit{actual reductions} are defined by
\begin{subequations}
\label{eq:ared}
\begin{align}
&\Delta f^{(k)}(d) \equaldef f^{(k)} - f(x^{(k)}+d),\\
&\Delta h^{(k)}(d) \equaldef h^{(k)} - h(x^{(k)}+d).
\end{align}
\end{subequations}

\subsubsection{Funnel Method}
\label{sec:funnel_mechanism}

A funnel (illustrated in Figure~\ref{fig:illustration_funnel_feasible_set}) describes a relaxation of the feasible set that allows a constraint violation up to a given upper bound $\tau^{(k)} > 0$. At iteration $k$, a necessary condition for acceptance of the trial iterate is the \textit{funnel condition}
\begin{equation}
\label{eq:funnel_condition}
h(x) \le \tau^{(k)}.
\end{equation}
The initial funnel width is given by
\begin{align}
\label{eq:funnel_initialization}
\tau^{(0)} = \max\left[\bar{\tau}, \bar{\kappa} h^{(0)} \right]
\end{align}
with $\bar{\tau} > 0$ and $\bar{\kappa} > 1$ to ensure that the initial point is acceptable. All iterates stay within the funnel whose width is monotonically non-increasing; that is, $\tau^{(k+1)} \leq \tau^{(k)}$ for all $k \geq 0$.
This property combined with a constraint relaxation strategy controls the feasibility of the iterates and ensures a feasible limit point.

\begin{figure}[htpb!]
\centering
% \documentclass{standalone}
% \usepackage{tikz}
% \usepackage{amsmath}

% \begin{document}

\begin{tikzpicture}

% Shaded areas
% \fill[gray!20] (-2.5,-2) rectangle (0,2);
% \fill[gray!50] (0,-2) rectangle (3,2);
% Define the paths
% Shaded area between the two curves
% Feasible set
\fill[gray!50, opacity=0.4] 
  plot[smooth,tension=1] coordinates {(-2.5,2) (-2.5,0) (-2.5,-2)}
  -- plot[smooth,tension=1] coordinates {(-2.5,2) (-1,0) (-2.5,-2)}
  -- cycle;

\fill[gray!50, opacity=0.4] 
  plot[smooth,tension=1] coordinates {(-2.5,2) (-1,0) (-2.5,-2)}
  -- plot[smooth,tension=1] coordinates {(0.5,-2) (1.5,0) (0.5, 2)}
  -- cycle;

% Curved lines
\draw[thick] plot[smooth,tension=1] coordinates {(-2.5,-2) (-1,0) (-2.5,2)};
\draw[thick,dashed] plot[smooth,tension=1] coordinates {(0,-2) (1,0) (0,2)};
\draw[thick] plot[smooth,tension=1] coordinates {(0.5,-2) (1.5,0) (0.5, 2)};

% Labels
\node at (-1.75,0) {\Large $\mathcal{F}$};
% funnel condition
\draw[thick] (0.9,-1.2) -- (1.5,-1.4);
\node at (2.5,-1.5) {\small $h^{(k)} \leq \tau^{(k)}$};
% funnel sufficient decrease condition
\draw[thick] (0.3,1.2) -- (1.4,1.4);
\node at (2.5,1.5) {\small $h^{(k)} \leq \beta \tau^{(k)}$};

\end{tikzpicture}

% \end{document}

%GAIL -- looks to me like the whole image is gray
% Charlie: yes, the funnel condition defines the gray region (it contains the feasible set F)
\caption{Funnel (in gray) around the feasible set $\mathcal{F}$. The frontier of the funnel envelope \eqref{eq:funnel_sufficient_decrease} is shown as a dashed curve.}
\label{fig:illustration_funnel_feasible_set}
\end{figure}
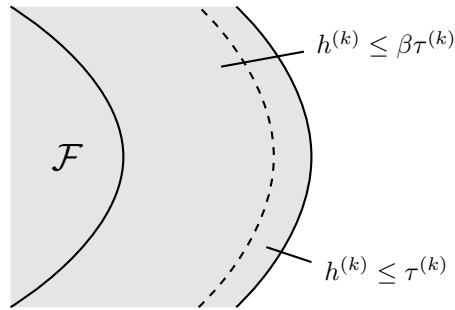

%If \eqref{eq:funnel_condition} is violated at the trial iterate $\trial{x}^{(k+1, l)}$, the trust-region radius or the step size is decreased.
Similarly to filter methods, a switching condition (with $\delta \in (0,1)$)
\begin{align}
\label{eq:switching_condition}
\Delta m_f^{(k)}(d) \geq \delta (h^{(k)})^2
\end{align}
ensures that the algorithm does not take infinitely small steps and thus avoids convergence toward infeasible points. It distinguishes between two types of iterations:
$f$-type iterations that improve optimality and $h$-type iterations that improve feasibility:
% We distinguish two classes of iterations depending on whether the switching condition \eqref{eq:switching_condition} is satisfied or not: 
\begin{itemize}
\item If the switching condition \eqref{eq:switching_condition} is satisfied, the trial iterate is accepted if an Armijo-type sufficient decrease condition
\begin{align}
\label{eq:sufficient_reduction}
\Delta f^{(k)} \geq \sigma \Delta m_f^{(k)}(d)
\end{align}
holds with $\sigma \in (0,1)$, that is, if the current iterate yields a sufficient decrease of the objective function. This is an $f$-type iteration.

\item If the switching condition \eqref{eq:switching_condition} is violated and if the funnel sufficient decrease condition
\begin{align}
\label{eq:funnel_sufficient_decrease}
h(\trial{x}^{(k+1, l)})\leq \beta \tau^{(k)}
\end{align}
holds with $\beta \in (0, 1)$, the trial iterate is accepted, and the funnel width is decreased:
\begin{align}
\label{eq:funnel_update}
\tau^{(k+1)} = (1 - \kappa) h(\trial{x}^{(k+1, l)}) + \kappa \tau^{(k)}
\end{align}
for $\kappa \in (0, 1)$. This is an $h$-type iteration.
\end{itemize}
Otherwise, if none of the above conditions is satisfied, the step is rejected, and either the trust-region radius or the step size is reduced. 
The complete funnel method is summarized in Algorithm~\ref{alg:acceptance} in the context of feasibility restoration. Here, $x^\mathit{resto}$ is the point at which the algorithm switches from the optimality phase to the feasibility restoration phase; it is set in Algorithms~\ref{alg:ls_funnel_restoration_sqp} and \ref{alg:direction_computation}.
% \textbf{GAIL - not sure why you put} in the preceding sentence. Do you use $x^\mathit{resto}$ anywhere?
% Charlie: it is used on line 6 of the algorithm
In the feasibility restoration phase, we simply enforce the Armijo condition on the constraint violation.

\begin{algorithm}[h!]
\caption{Restoration funnel acceptance test}
\label{alg:acceptance}
\scriptsize
\KwIn{trial iterate $\trial{x}^{(k+1)}$, direction $d^{(k,l)}$}
\begin{constraint_relaxation}
$\mathit{acceptable} \gets \mathit{false}$ \;
\eIf{$\| d^{(k,l)} \| = 0$} {
    $\mathit{acceptable} \gets \mathit{true}$\tcp*[f]{KKT point found}\;
}{
    % possibly switch to optimality
    \If{$\mathit{phase}$ = Restoration and \colorbox{subproblem_color}{subproblem} feasible and $\highlight[strategy_color]{h(\trial{x}^{(k+1, l)}) \le \beta \min(\tau^{(k)}, h(x^\mathit{resto}))}$} {
        $\mathit{phase} \gets$ Optimality \;
        \begin{globalization_strategy}
        $\tau^{(k+1)} \gets (1 - \kappa) h(\trial{x}^{(k+1)}) + \kappa \tau^{(k)}$ \;
        \end{globalization_strategy}
    }

    \uIf{$\mathit{phase}$ = Optimality} {
        \begin{globalization_strategy}
        \If{trial iterate is acceptable to funnel: $h(\trial{x}^{(k+1)}) \leq \tau^{(k)}$}{
            \uIf(\tcp*[f]{f-type step}){switching condition is satisfied: $\Delta m_f^{(k)}(d) \ge \delta \left(h^{(k)}\right)^2$}{
                \If{Armijo condition is satisfied: $f^{(k)} - f(\trial{x}^{(k+1)}) \geq \sigma \Delta m_f^{(k)}(d)$} {
                $\mathit{acceptable} \gets \mathit{true}$ \;}
            }
            \ElseIf(\tcp*[f]{h-type step}){funnel is sufficiently reduced: $h(\trial{x}^{(k+1)}) \leq \beta \tau^{(k)}$}{
                $\mathit{acceptable} \gets \mathit{true}$ \;
                $\tau^{(k+1)} \gets (1 - \kappa) h(\trial{x}^{(k+1)}) + \kappa \tau^{(k)}$\;
            }
        }
        \end{globalization_strategy}
    }
    \ElseIf {$\mathit{phase}$ = Restoration}{
        \begin{globalization_strategy}
        \If{Armijo condition is satisfied: $h^{(k)} - h(\trial{x}^{(k+1)}) \ge \sigma \Delta m_h^{(k)}(d)$}{
            $\mathit{acceptable} \gets \mathit{true}$ \;
            % \If{$h(\trial{x}^{(k+1)}) \le \beta \tau^{(k)}$ and $h(\trial{x}^{(k+1)}) \le \beta h(x^\mathit{resto})$}{
            % $\mathit{phase} \gets \mathit{optimality}$ \;
            % $\tau^{(k+1)} \gets (1 - \kappa) h(\trial{x}^{(k+1)}) + \kappa \tau^{(k)}$ \;
            % }
        }
        \end{globalization_strategy}
    }
}
\end{constraint_relaxation}
\Return $\mathit{acceptable}$ \;
\end{algorithm}

Inspired by \cite{Samadi2018}, the principle of the funnel is illustrated in Figure~\ref{fig:illustration_funnel_mechanism}.
In Figure~\ref{fig:funnel_satisfaction}, the funnel width is the solid vertical line, and the funnel envelope is represented by the dotted line. The black dot is the current iterate, the green dots show possible acceptable iterates, and the red dot lies outside of the funnel and is rejected.
%\textbf{GAIL - I know because you mention two green dots} that the two dots are green but honestly, they don't really look green unless one enlarges the page a lot. (Maybe when one prints thepage, they do look green?
% Charlie: fixed
Figure~\ref{fig:funnel_h_type} represents an $h$-type iteration: the switching condition is not satisfied, but the funnel sufficient decrease condition is satisfied. The trial iterate is accepted, and the funnel width is decreased according to \eqref{eq:funnel_update}. This guarantees that the feasibility of the iterates is driven to zero. The previous funnel width is shown in gray.
Figure~\ref{fig:funnel_f_type} represents an $f$-type iteration: both the switching condition and the Armijo sufficient decrease condition are satisfied. The trial iterate is accepted, but the funnel width is not updated.
Figure~\ref{fig:funnel_convergence} shows the convergence of the funnel method to the optimal solution.

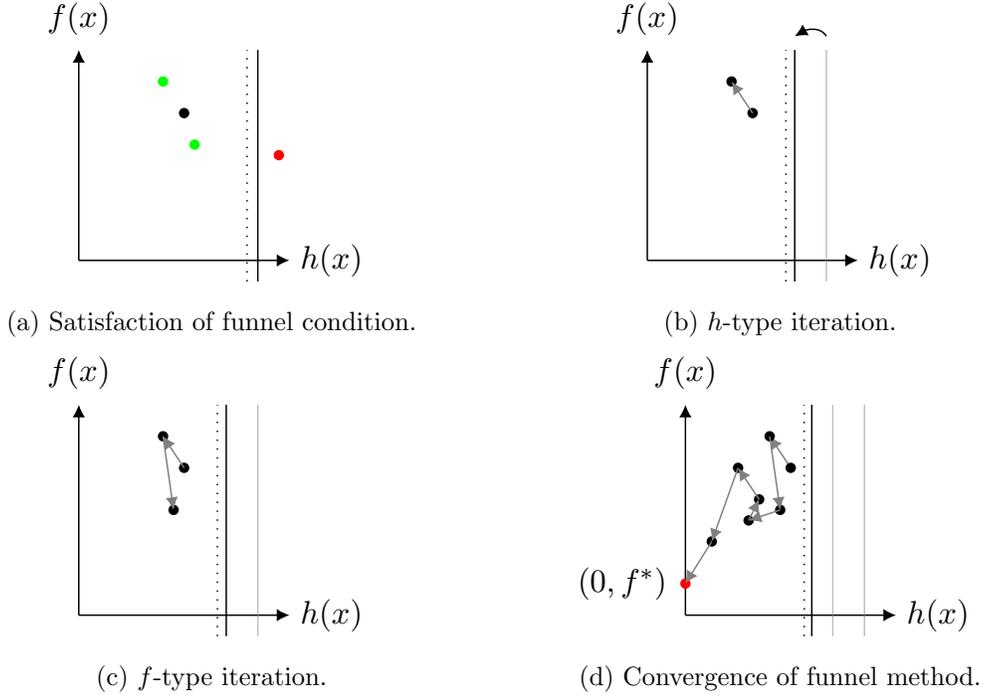
\begin{figure}[htbp!]
\centering
\def\scaling{1.4}
% Upper Left Figure
\begin{subfigure}{.45\linewidth}
\centering
\scalebox{\scaling}{\begin{tikzpicture}
\draw[-{Latex[length=1.2mm, width=1.2mm]}] (0.0, 0.0) -- (2.0, 0.0) node[left] {};
\draw[-{Latex[length=1.2mm, width=1.2mm]}] (0.0, 0.0) -- (0.0, 2.0) node[left] {};

\node (A) at (0.0, 2.3) {\scalebox{.8}{$f(x)$}};
\node (A) at (2.4, 0.0) {\scalebox{.8}{$h(x)$}};

% iterates definition
\filldraw (1, 1.4) circle (1.2pt) node[left] {};
\filldraw[color=green] (0.8, 1.7) circle (1.2pt) node[left] {};
\filldraw[color=green] (1.1, 1.1) circle (1.2pt) node[left] {};
\filldraw[color=red] (1.9, 1.0) circle (1.2pt) node[left] {};

% funnel parameters
\def\funnelwidth{1.7}
\def\envelope{0.94}

% funnel width
\draw[solid] (\funnelwidth, -0.2) -- (\funnelwidth, 2) node[left] {};
% funnel envelope
\draw[dotted] (\envelope*\funnelwidth, -0.2) -- (\envelope*\funnelwidth, 2) node[left] {};
\end{tikzpicture}}
\caption{Satisfaction of funnel condition.}
\label{fig:funnel_satisfaction}
\end{subfigure}
%\hfill
% Upper Right Figure
\begin{subfigure}{.45\linewidth}
\centering
\scalebox{\scaling}{\begin{tikzpicture}
% The coordinate system
% \begin{scope}[transparency group, opacity=0.75]
\draw[-{Latex[length=1.2mm, width=1.2mm]}] (0.0, 0.0) -- (2.0, 0.0) node[left] {};
\draw[-{Latex[length=1.2mm, width=1.2mm]}] (0.0, 0.0) -- (0.0, 2.0) node[left] {};

\node (A) at (0.0, 2.3) {\scalebox{.8}{$f(x)$}};
\node (A) at (2.4, 0.0) {\scalebox{.8}{$h(x)$}};

% iterates definition
\coordinate (x0) at (1, 1.4);
\coordinate (x1) at (0.8, 1.7);

\filldraw (x0) circle (1.2pt) node[left] {};
\filldraw (x1) circle (1.2pt) node[left] {};
\draw[-{Latex[length=1.2mm, width=1.2mm]}, gray] (x0) -- (x1) node[left] {};

% funnel parameters
\def\oldfunnelwidth{1.7}
\def\funnelwidth{1.4}
\def\envelope{0.94}

% old funnel width
\draw[solid, color=gray, opacity=0.5] (\oldfunnelwidth, -0.2) -- (\oldfunnelwidth, 2) node[left] {};
% funnel witdh
\draw[solid] (\funnelwidth, -0.2) -- (\funnelwidth, 2) node[left] {};
% funnel envelope
\draw[dotted] (\envelope*\funnelwidth, -0.2) -- (\envelope*\funnelwidth, 2) node[left] {};
% arrow
\draw[-{Latex[length=1.2mm, width=1.2mm]}] (\oldfunnelwidth, 2.13) to [out=120,in=30] (\funnelwidth, 2.13) node[left] {};
\end{tikzpicture}}
\caption{$h$-type iteration.}
\label{fig:funnel_h_type}
\end{subfigure}
% Lower Left Figure
\begin{subfigure}{.45\linewidth}
\centering
\scalebox{\scaling}{\begin{tikzpicture}
\draw[-{Latex[length=1.2mm, width=1.2mm]}] (0.0, 0.0) -- (2.0, 0.0) node[left] {};
\draw[-{Latex[length=1.2mm, width=1.2mm]}] (0.0, 0.0) -- (0.0, 2.0) node[left] {};

\node (A) at (0.0, 2.3) {\scalebox{.8}{$f(x)$}};
\node (A) at (2.4, 0.0) {\scalebox{.8}{$h(x)$}};

% iterates definition
\coordinate (x0) at (1, 1.4);
\coordinate (x1) at (0.8, 1.7);
\coordinate (x2) at (0.9, 1.0);

\filldraw (x0) circle (1.2pt) node[left] {};
\filldraw (x1) circle (1.2pt) node[left] {};
\filldraw (x2) circle (1.2pt) node[left] {};

\draw[-{Latex[length=1.2mm, width=1.2mm]}, gray] (x0) -- (x1) node[left] {};
\draw[-{Latex[length=1.2mm, width=1.2mm]}, gray] (x1) -- (x2) node[left] {};

% funnel parameters
\def\oldfunnelwidth{1.7}
\def\funnelwidth{1.4}
\def\envelope{0.94}

% old funnel width
\draw[solid, color=gray, opacity=0.5] (\oldfunnelwidth, -0.2) -- (\oldfunnelwidth, 2) node[left] {};
% funnel witdh
\draw[solid] (\funnelwidth, -0.2) -- (\funnelwidth, 2) node[left] {};
% funnel envelope
\draw[dotted] (\envelope*\funnelwidth, -0.2) -- (\envelope*\funnelwidth, 2) node[left] {};
\end{tikzpicture}}
\caption{$f$-type iteration.}
\label{fig:funnel_f_type}
\end{subfigure}
%\hfill
% Lower Right Figure
\begin{subfigure}{.45\linewidth}
\centering
\scalebox{\scaling}{\begin{tikzpicture}
\draw[-{Latex[length=1.2mm, width=1.2mm]}] (0.0, 0.0) -- (2.0, 0.0) node[left] {};
\draw[-{Latex[length=1.2mm, width=1.2mm]}] (0.0, 0.0) -- (0.0, 2.0) node[left] {};

\node (A) at (0.0, 2.3) {\scalebox{.8}{$f(x)$}};
\node (A) at (2.4, 0.0) {\scalebox{.8}{$h(x)$}};

% iterates definition
\coordinate (x0) at (1, 1.4);
\coordinate (x1) at (0.8, 1.7);
\coordinate (x2) at (0.9, 1.0);
\coordinate (x3) at (0.6, 0.9);
\coordinate (x4) at (0.7, 1.1);
\coordinate (x5) at (0.5, 1.4);
\coordinate (x6) at (0.25, 0.7);
\coordinate (xopt) at (0, 0.3);

\draw[-{Latex[length=1.2mm, width=1.2mm]}, gray] (x0) -- (x1) node[left] {};
\filldraw (x0) circle (1.2pt) node[left] {};
\filldraw (x1) circle (1.2pt) node[left] {};
\filldraw (x2) circle (1.2pt) node[left] {};
\filldraw (x3) circle (1.2pt) node[left] {};
\filldraw (x4) circle (1.2pt) node[left] {};
\filldraw (x5) circle (1.2pt) node[left] {};
\filldraw (x6) circle (1.2pt) node[left] {};
\filldraw[color=red] (xopt) circle (1.2pt) node[left] {};
\node (A) at (-0.6, 0.3) {\scalebox{.8}{$(0, f^*)$}};

\draw[-{Latex[length=1.2mm, width=1.2mm]}, gray] (x1) -- (x2) node[left] {};
\draw[-{Latex[length=1.2mm, width=1.2mm]}, gray] (x2) -- (x3) node[left] {};
\draw[-{Latex[length=1.2mm, width=1.2mm]}, gray] (x3) -- (x4) node[left] {};
\draw[-{Latex[length=1.2mm, width=1.2mm]}, gray] (x4) -- (x5) node[left] {};
\draw[-{Latex[length=1.2mm, width=1.2mm]}, gray] (x5) -- (x6) node[left] {};
\draw[-{Latex[length=1.2mm, width=1.2mm]}, gray] (x6) -- (xopt) node[left] {};

% funnel parameters
\def\veryoldfunnelwidth{1.7}
\def\oldfunnelwidth{1.4}
\def\funnelwidth{1.2}
\def\envelope{0.94}

% very old funnel width
\draw[solid, color=gray, opacity=0.5] (\veryoldfunnelwidth, -0.2) -- (\veryoldfunnelwidth, 2) node[left] {};
% old funnel width
\draw[solid, color=gray, opacity=0.5] (\oldfunnelwidth, -0.2) -- (\oldfunnelwidth, 2) node[left] {};
% funnel witdh
\draw[solid] (\funnelwidth, -0.2) -- (\funnelwidth, 2) node[left] {};
% funnel envelope
\draw[dotted] (\envelope*\funnelwidth, -0.2) -- (\envelope*\funnelwidth, 2) node[left] {};
\end{tikzpicture}}
\caption{Convergence of funnel method.}
\label{fig:funnel_convergence}
\end{subfigure}
\caption{Illustration of the funnel method~\cite{Samadi2018}.}
\label{fig:illustration_funnel_mechanism}
\vspace{-3mm}
\end{figure}

\paragraph{A note on the funnel reduction mechanism.}
In the first paper on funnel methods \cite{gould2010}, the funnel update was given by 
\begin{align}
\label{eq:gould_funnel_update}
\tau^{(k+1)}=\max \left[\beta \tau^{(k)}, (1-\kappa)h(\trial{x}^{(k+1, l)})
 +\kappa h^{(k)}\right]. 
\end{align}
This keeps the iterates within the funnel if $h(\trial{x}^{(k+1, l)})\leq h^{(k)}$. In our case, this cannot be guaranteed since we also allow steps that increase both optimality and infeasibility. We tried out different update strategies inspired by \cite{gould2010}, but all had similar performance. Therefore, we opted for \eqref{eq:funnel_update}, which is the second term in the $\max$. More important is the choice of the parameter $\kappa$. If $\kappa$ is large such that the funnel width is slowly decreased, the funnel method allows for too much non-monotonicity, and performance can be degraded. We note that we decrease the funnel width only  in $h$-type iterations.

\subsubsection{Filter Methods}

Using the notion of \cite{waechter2005},  we can define a filter  as a taboo region in the $\{(h, f) \in \R^2 \colon~ h \geq 0 \}$ half-plane, defined by a list of points $(h(x_p), f(x_p))$ (typically, previous iterates). The filter at iteration $k$ is denoted by $\mathcal{F}^{(k)}$. A trial iterate $\trial{x}^{(k+1,l)}$ is acceptable to the filter if it sufficiently decreases feasibility or optimality (or both), 
\begin{align*}
h(\trial{x}^{(k+1,l)}) \le \beta h(x_p)
\quad \text{or} \quad
f(\trial{x}^{(k+1,l)}) \le f(x_p) - \gamma h(\trial{x}^{(k+1,l)}) \quad \text{for all} \quad (h(x_p), f(x_p)) \in \mathcal{F}^{(k)},
\end{align*}
where $\beta, \gamma>0$. The first condition is similar to \eqref{eq:funnel_sufficient_decrease}, and the second condition is a sufficient decrease of the objective. 
During the optimization process, it is ensured that every iterate does not lie within the taboo region.
Often the filter is initialized with an upper bound $h_{\max}$ on the constraint violation that mimics a filter entry $(h_{\max}, -\infty)$.
The acceptability of a trial iterate with respect to $h_{\max}$ is given by
\begin{align}
h(\trial{x}^{(k+1, l)})\leq \beta h_{\max}.
\end{align}

Filter methods also distinguish between $f$-type and $h$-type iterations by means of a switching condition. In the case of an $f$-type iteration, an Armijo condition checks for sufficient decrease, and the filter remains unchanged.
In $h$-type iterations, the trial iterate is accepted if it satisfies sufficient reduction with respect to the
% filter
current iterate:
\begin{align*}
f(\trial{x}^{(k+1)}) \leq f^{(k)} - \gamma h(\trial{x}^{(k+1)}) \quad \text{or}\quad h(\trial{x}^{(k+1)}) \leq \beta h^{(k)},
\end{align*}
in which case the current iterate is added to the filter. This prevents the algorithm from cycling.
We note that in the unconstrained case, we have to take more care to accept steps, because ${h(\trial{x}^{(k+1)}) \leq \beta h^{(k)}}$ holds trivially in that case, and we instead enforce a sufficient reduction condition on $f$ as usual for unconstrained optimization.

\subsubsection{Connections between Funnel and Filter Methods}

An interesting connection exists between filter and funnel methods. In practical implementations of filter methods, we limit the number of filter entries to a maximum number, say $N_{\cal F}$, and initialize the filter with an upper bound on the constraint violation $h_{\max} = \max\{\kappa_1 h(x^{(0)}), \kappa_2\}$, where $\kappa_1, \kappa_2 >1$ are constants that ensure that the initial point $x^{(0)}$ is acceptable. If the capacity $N_{\cal F}$ of the filter is reached, we reduce the upper bound by replacing it with an upper bound derived from the filter entry that has the maximum constraint violation $(h^+, f^+) \in {\cal F}$, resulting in a new upper bound $(h^+,-\infty)$.  This approach frees one filter entry and ensures that the current point remains filter-acceptable, and we cannot cycle.

We can now interpret the funnel as a filter with a single entry,  $N_{\cal F}=1$, that corresponds to the upper bound on the constraint violation, and we update this upper bound on h-type iterations. The main difference between this interpretation and the actual funnel method is  the update rule \eqref{eq:funnel_update}.
% We argue that a funnel method can be interpreted as a filter method that contains a single filter entry that corresponds to an upper bound on the constraint violation.
This interpretation allows us to provide a streamlined convergence proof.

\section{A Unified Funnel Restoration SQP Algorithm}
\label{sec:funnel_restoration_sqp}

%\begin{algorithm}[htb]
%\caption{\colorbox{mechanism_color}{Globalization Mechanism: Computation of Search Direction}}
%\label{alg:funnel_restoration_phase}
%% \SetKwInOut{Parameter}{Parameter}
%% \Parameter{$\beta, \kappa, \delta, \sigma$}
%\SetKwProg{myproc}{Procedure}{}{}
%\myproc{$\mathrm{AssembleTrialIterate}(x^{(k)}, \mathit{mechanism}, \mathit{phase})$}{
    %% $\mathit{acceptable} \gets \mathit{false}; \mathit{switch}\gets \mathit{false}$, $\tau^{(k+1)}\gets \tau^{(k)}$\;
    %\uIf{$\mathit{phase}==\mathit{optimality}$} {
        %\lIf{$\mathit{mechanism}==\mathit{trust-region}$}
        %{solve \eqref{eq:TR-QP}}
        %\lElseIf{$\mathit{mechanism}==\mathit{line-search}$}{solve \eqref{eq:QP-subproblem}}
        %\If{QP infeasible or step size below minimum}{
            %$\mathit{phase}\gets\mathit{restoration}$\;
            %% restart $\mathrm{AssembleTrialIterate}(x^{(k)}, \mathit{mechanism}, \mathit{phase})$
        %}
    %}
    %\If{$\mathit{phase}==\mathit{restoration}$} {
        %\lIf{$\mathit{mechanism}==\mathit{trust-region}$}
        %{solve \eqref{eq:FQP}}
        %\lElseIf{$\mathit{mechanism}==\mathit{line-search}$}{solve $(FQP(x^{(k)}))$}
    %}
    %\Return  $\mathit{acceptable}$, $\mathit{phase}$\;
%}
%\end{algorithm}

The complete method is presented in Algorithms~\ref{alg:acceptance}, \ref{alg:tr_funnel_restoration_sqp}, \ref{alg:ls_funnel_restoration_sqp}, and \ref{alg:direction_computation}. It follows the double-loop framework of Algorithm~\ref{alg:frameworkUNO}: the outer loop updates the current iterate and the inner loop computes an acceptable point (either in the optimality phase or in the feasibility restoration phase).
Algorithm \ref{alg:acceptance} implements the funnel globalization strategy; Algorithms~\ref{alg:tr_funnel_restoration_sqp} and \ref{alg:ls_funnel_restoration_sqp} implement the double-loop SQP with a trust-region and with a line-search globalization mechanism, respectively; and, Algorithm~\ref{alg:direction_computation} implements the direction computation. We deliberately split the two methods into separate components to emphasize the modularity of the \texttt{Uno} framework.

\begin{algorithm}[h!]
\caption{\texttt{Uno}: trust-region funnel restoration SQP method.}
\label{alg:tr_funnel_restoration_sqp}
\scriptsize
\SetAlgoVlined
\KwIn{initial primal-dual iterate $(x^{(0)},\constraintmultipliers^{(0)}, \boundmultipliers^{(0)})$}
$k \gets 0$ \;
$\mathit{termination} \gets \mathit{false}$ \;
\begin{constraint_relaxation}
$\mathit{phase} \gets$ Optimality \;
\end{constraint_relaxation}
\begin{globalization_strategy}
$\tau^{(0)} = \max \left[\bar{\tau},\bar{\kappa} h^{(0)}\right]$
\end{globalization_strategy}
\Repeat{$\mathit{termination}$} {
    \begin{globalization_mechanism}
    Set inner iteration counter $l \gets 0$ \;
    \Repeat{$(\trial{x}^{(k+1, l)}, \trial{\constraintmultipliers}^{(k+1, l)}, \trial{\boundmultipliers}^{(k+1, l)})$ is acceptable} {
        \begin{constraint_relaxation}
        Compute primal-dual solution $(d^{(k,l)}, \trial{\constraintmultipliers}^{(k+1)}, \trial{\boundmultipliers}^{(k+1)})$ (Algorithm~\ref{alg:direction_computation}) \;
        \end{constraint_relaxation}
        Assemble trial iterate $\trial{x}^{(k+1,l)} \equaldef x^{(k)} + d^{(k, l)}$ \;
        Reset the multipliers corresponding to the active trust region \;
        \begin{constraint_relaxation}
        Determine whether trial iterate is $\mathit{acceptable}$ (Algorithm~\ref{alg:acceptance}) \;
        \end{constraint_relaxation}
        \eIf{$acceptable$} {
            \lIf{trust region is active at $d^{(k,l)}$} {
                increase radius $\TRradius^{(l)}$
            }
        }{
            Decrease radius $\TRradius^{(l)}$ \;
            $l \gets l+1$ \;
        }
    }
    \end{globalization_mechanism}
    Update $(x^{(k+1)}, \constraintmultipliers^{(k+1)}, \boundmultipliers^{(k+1}) \gets (\trial{x}^{(k+1, l)}, \trial{\constraintmultipliers}^{(k+1, l)}, \trial{\boundmultipliers}^{(k+1, l)})$ \;
    \lIf{termination criteria satisfied at $(x^{(k+1)}, \constraintmultipliers^{(k+1)}, \boundmultipliers^{(k+1})$}{$\mathit{termination \gets true}$}
	$k \gets k+1$ \;
}
\Return $(x^{(k)}, \constraintmultipliers^{(k)}, \boundmultipliers^{(k)})$
\end{algorithm}

\begin{algorithm}[h!]
\caption{\texttt{Uno}: line-search funnel restoration SQP method.}
\label{alg:ls_funnel_restoration_sqp}
\scriptsize
\SetAlgoVlined
\KwIn{initial primal-dual iterate $(x^{(0)},\constraintmultipliers^{(0)}, \boundmultipliers^{(0)})$}
$k \gets 0$ \;
$\mathit{termination} \gets \mathit{false}$ \;
\begin{constraint_relaxation}
$\mathit{phase} \gets$ Optimality \;
\end{constraint_relaxation}
\begin{globalization_strategy}
$\tau^{(0)} = \max \left[\bar{\tau},\bar{\kappa} h^{(0)}\right]$
\end{globalization_strategy}
\Repeat{$\mathit{termination}$} {
    \begin{globalization_mechanism}
    \begin{constraint_relaxation}
    Compute primal-dual solution $(d^{(k)}, \trial{\constraintmultipliers}^{(k+1)}, \trial{\boundmultipliers}^{(k+1)})$ (Algorithm~\ref{alg:direction_computation}) \;
    \end{constraint_relaxation}
    $\alpha^{(0)} \gets 1$ \;
    Set inner iteration counter $l \gets 0$ \;
    
    \Repeat{$(\trial{x}^{(k+1, l)}, \trial{\constraintmultipliers}^{(k+1, l)}, \trial{\boundmultipliers}^{(k+1, l)})$ is acceptable} {
        Assemble trial iterate $\trial{x}^{(k+1,l)} \equaldef x^{(k)} + \alpha^{(l)} d^{(k)}$ \;
        \begin{constraint_relaxation}
        Determine whether trial iterate is $\mathit{acceptable}$ (Algorithm~\ref{alg:acceptance}) \;
        \end{constraint_relaxation}
        \If{not $acceptable$} {
            Decrease step length $\alpha^{(l)}$ \;
            $l \gets l+1$ \;
        }
        \If{$\alpha^{(l)}<\alpha_{\min}$ (step-size too small)} {
            \begin{constraint_relaxation}
            $\mathit{phase} \gets$ Restoration \;
            $x^\mathit{resto} \gets x^{(k)}$ \;
            Compute primal-dual solution $(d^{(k)}, \trial{\constraintmultipliers}^{(k+1)}, \trial{\boundmultipliers}^{(k+1)})$ (Algorithm~\ref{alg:direction_computation}) \;
            \end{constraint_relaxation}
            $\alpha^{(l)} \gets 1$ \;
        }
    }
    \end{globalization_mechanism}
    Update $(x^{(k+1)}, \constraintmultipliers^{(k+1)}, \boundmultipliers^{(k+1}) \gets (\trial{x}^{(k+1, l)}, \trial{\constraintmultipliers}^{(k+1, l)}, \trial{\boundmultipliers}^{(k+1, l)})$ \;
    \lIf{termination criteria satisfied at $(x^{(k+1)}, \constraintmultipliers^{(k+1)}, \boundmultipliers^{(k+1})$}{$\mathit{termination \gets true}$}
	$k \gets k+1$ \;
}
\Return $(x^{(k)}, \constraintmultipliers^{(k)}, \boundmultipliers^{(k)})$
\end{algorithm}

\begin{algorithm}[h!]
\caption{\texttt{Uno}: QP direction computation}
\label{alg:direction_computation}
\scriptsize
\SetAlgoVlined
\KwIn{current iterate $(x^{(k)}, \constraintmultipliers^{(k)}, \boundmultipliers^{(k)})$}
\begin{constraint_relaxation}
\If{$\mathit{phase}$ = Optimality} {
    $(d^{(k,l)}, \trial{\constraintmultipliers}^{(k+1)}, \trial{\boundmultipliers}^{(k+1)}) \gets$ solve \colorbox{subproblem_color}{subproblem \ref{eq:QP-subproblem} or \ref{eq:TR-QP}} \;
    
    % switch to feasibility restoration
    \If{\colorbox{subproblem_color}{subproblem} infeasible} {
        $\mathit{phase} \gets$ Restoration \;
        $x^\mathit{resto} \gets x^{(k)}$ \;
        $\constraintmultipliers^{(k)} \gets 0$ \;
    }
}
\If{$\mathit{phase}$ = Restoration} {
    $(d^{(k,l)}, \trial{\constraintmultipliers}^{(k+1)}, \trial{\boundmultipliers}^{(k+1)}) \gets$ solve \colorbox{subproblem_color}{feasibility subproblem $(FQP(x^{(k)}))$ or \ref{eq:FQP}} \; % starting from $d^{(k,l)}$ \;
}
\end{constraint_relaxation}
\Return $(d^{(k,l)}, \trial{\constraintmultipliers}^{(k+1)}, \trial{\boundmultipliers}^{(k+1)})$
\end{algorithm}
\section{Global Convergence of Trust-Region Funnel SQP Algorithm}
\label{sec:global_convergence_proofs}

%\todo[inline,color=red!10]{Sven: I don't believe it is easy to unify the convergence proof, see more comments below (in the same colored boxes), so I am moving back to a TR proof only.}

% \paragraph{ToDo:}
% \begin{itemize}[nolistsep]
% \item Try proof with inexact QP solves, or at least remove need for global min of QP ... \textcolor{red}{why is this needed?} David: I meant the proof of Philippe Toint with that
% \item Look at proof by Philippe Toint: normal and tangential step decomposition; Cauchy step.
% \end{itemize}\bigskip

In this section we prove that the trust-region funnel SQP method converges either to a feasible point, which is a KKT point if a constraint qualification holds, or to a stationary point of the constraint violation.
% ensures convergence of the iterates to a feasible point, then prove global convergence.
We start by stating our basic assumptions.
\begin{assumption}[Standard assumptions~\cite{fletcher1997}]\ \\
\label{assumption:unified_feasibility_convergence}
\vspace{-6mm}
\begin{enumerate}
\item All points attained by the funnel algorithm lie in a nonempty compact set $X$.
\item The functions $f$ and $c$ are twice continuously differentiable on an open set containing $X$.
\item The Hessian matrices of the objective and constraints are uniformly bounded for all $x \in X$; in other words, there exists an $M > 0$ such that the Hessian matrices $H(x)$ satisfy $\Vert H(x) \Vert_2 \leq M$ for all $x \in X$.
\end{enumerate}
\end{assumption}
In particular, these assumptions ensure that the problem functions are bounded and that the objective function $f$ is bounded from below.

\subsection{Convergence to Feasibility}

The convergence to a feasible point is independent of whether we use a trust-region or a line search mechanism and instead relies on the switching condition, the boundedness of $f$, and the funnel mechanism.
In Theorem~\ref{theorem:convergence_to_feasibility} we prove that the algorithm generates a sequence of iterates that converges
%\todo{in English, do we say that the points converge or that the sequence converges? This should be made consistent throughout. We can use both, but seqence is better} 
toward feasibility.
This theorem implicitly assumes that the algorithm is well defined (i.e., that the inner iteration terminates) and generates an infinite sequence, which is shown below in Lemma~\ref{lemma:finite_inner_loop} for the trust-region method (a similar result follows easily for the line-search method because we switch to a restoration phase once the step size $\alpha$ becomes smaller than $\alpha_{\min}$).

\begin{theorem}
\label{theorem:convergence_to_feasibility}
Assume that the algorithm does not terminate finitely at a KKT point, and consider sequences $\{\tau^{(k)}\}$, $\{h^{(k)}\}$, and $\{f^{(k)}\}$ such that $h^{(k)} \geq 0$, $f^{(k)}$ is bounded below and $h^{(k+n)} \leq \beta \tau^{(k)}$ for all $k, n \in \N$. Furthermore, let constants $\beta \in (0, 1)$, $\kappa \in (0, 1)$ be given. It holds either that
\begin{align}
\label{eq:h_type}
\text{($h$-type step)} \quad
h^{(k+1)} \leq \beta \tau^{(k)} \quad \text{and} \quad \tau^{(k+1)} = (1 - \kappa) h^{(k+1)} + \kappa \tau^{(k)}
\end{align}
or that
\begin{align}
\label{eq:f_type}
\text{($f$-type step)} \quad
\tau^{(k+1)} = \tau^{(k)} \quad \text{and} \quad f^{(k)} - f^{(k+1)} \geq \sigma \delta (h^{(k)})^2.
\end{align}
In both cases, it follows that $h^{(k)} \to 0$ for $k \to \infty$.
\end{theorem}

\begin{proof}
We study two cases, depending on whether there exists an infinite sequence of $h$-type steps.
\begin{enumerate}
\item If the number of $h$-type updates \eqref{eq:h_type} is infinite,
% The number of $f$-type updates \eqref{eq:f_type} can be infinite or finite.
% Then, there exists a strictly monotonically decreasing subsequence of $\{\tau^{(k)}\}$.
%In particular, for all $\varepsilon > 0$, there exists a $\overline{k}$ such that for all $k>\overline{k}$ of the original sequence it holds $\tau^{(k)}<\varepsilon$ and $h^{(k+n)} \leq \tau^{(k)}$ for all $n\in\N$. Therefore, it follows $h^{(k)} \to 0$.
the funnel update rule implies that
\begin{align*}
\tau^{(k+1)} & = (1 - \kappa) h^{(k+1)} + \kappa \tau^{(k)} \\
            & \le (1 - \kappa) \beta \tau^{(k)} + \kappa \tau^{(k)} \\
            & \le \left( 1 - (1 - \beta)(1 - \kappa) \right) \tau^{(k)}.
\end{align*}
Therefore, $\tau^{(k+1)} \le \theta \tau^{(k)}$, where $\theta \equaldef 1 - (1 - \beta)(1 - \kappa) \in (0, 1)$. Thus $\tau^{(k)} \to 0$ for $k \to \infty$.

\item If the number of $h$-type steps is finite, there exists a $\overline{k}\in\N$ such that for all $k \geq \overline{k}$ only $f$-type updates are performed, that is, both the switching condition ${\Delta m_f^{(k)}(d) \geq \delta (h^{(k)})^2}$ and the Armijo sufficient decrease condition $\Delta f^{(k)} \geq \sigma \Delta m_f^{(k)}(d)$ are satisfied. In the case of the line-search mechanism we have $\Delta f^{(k)} \geq \alpha \sigma \Delta m_f^{(k)}(d) > \alpha_{\min} \sigma \Delta m_f^{(k)}(d)$, because the step was not a restoration step. %\eqref{eq:f_type} follows.
% From \eqref{eq:f_type}, we see $f^{(k+1)}\leq f^{(k)}$.
We then sum the switching condition \eqref{eq:f_type} over $k$ from $\overline{k}$ to a given $\overline{k} + N$ and observe that
\begin{align*}
\sum_{k=\overline{k}}^{\overline{k}+N} \left( f^{(k)} - f^{(k+1)} \right) \geq \sigma \delta \sum_{k=\overline{k}}^{\overline{k}+N} (h^{(k)})^2,
\end{align*}
which is equivalent to
\begin{align*}
f^{(\overline{k})} - f^{(\overline{k}+N)} \geq \sigma \delta \sum_{k=\overline{k}}^{\overline{k}+N} (h^{(k)})^2,
\end{align*}
where the right-hand side is multiplied by $\alpha_{\min}>0$ for the line-search mechanism.
For $N \to \infty$, $\displaystyle \sum_{k=\overline{k}}^{\overline{k} + N} (h^{(k)})^2$ is bounded (because $f$ is bounded below by Assumption~\ref{assumption:unified_feasibility_convergence}), and therefore $h^{(k)}\to 0$.
\end{enumerate}
\end{proof}

\subsection{Global Convergence of Trust-Region Funnel Method to Stationary Points}

The remainder of the global convergence analysis for trust-region and line-search methods differs significantly, and here we  analyze only trust-region methods.

The trust-region funnel SQP method has three possible outcomes: (1) it terminates finitely or converges in the limit at a KKT point that satisfies MFCQ; (2) it converges to a Fritz John point at which MFCQ fails; or (3) it converges to a stationary point of the constraint violation. Outcome (3) is as strong as we can hope for in some sense, because finding a feasible point of \eqref{eq:NCO} is just as hard as finding a stationary point.

In our analysis we concentrate on outcomes (1) and (2) and  consider only infinite sequences. 
%%SL We want to find KKT points that satisfy the MFCQ condition. 
Similar to the global convergence paper of the filter method~\cite{fletcher2002a}, we first show properties in a neighborhood of the feasible set. We start with two lemmas proven in \cite{fletcher2002a}
%\todo{it rather seems to be in \cite{fletcher2002a}} 
(due to a different definition of \eqref{eq:NCO}, we give an adjusted proof).
\begin{lemma}[Lemma 2 in \cite{fletcher2002a}]
\label{lemma:1Dminimization}
Consider minimizing a quadratic function $\phi(\alpha)$, $\phi \colon \R \rightarrow \R$ on the interval $\alpha \in [0, 1]$ when $\phi^{\prime}(0) < 0$. A necessary and sufficient condition for the minimizer to be at $\alpha = 1$ is $\phi^{\prime \prime}+\phi^{\prime}(0) \leq 0$. In this case it follows that $\phi(0) - \phi(1) \geq$ $-\frac{1}{2} \phi^{\prime}(0)$.
\end{lemma}

\begin{lemma}
Let the standard assumptions hold, and let $d$ be a feasible point of \ref{eq:TR-QP}. It then follows that
\begin{align}    
& \Delta f^{(k)} \geq \Delta m_f^{(k)} - n (\TRradius^{(l)})^2 M, \label{eq:ared_pred_relation} \\
&\left| c_i \left( x^{(k)} + d \right) \right| \le \frac{1}{2} n (\TRradius^{(l)})^2 M, \quad i \in \{1, \ldots, m\}.
\label{eq:constraint_bounds}
\end{align}
\end{lemma}
\begin{proof}
Relation \eqref{eq:ared_pred_relation} is derived in \cite{fletcher2002a}. 
% From the intermediate value form of Taylor's theorem, there exists $y$ on the line segment from $x^{(k)}$ to $x^{(k)}+d$ such that:
% \begin{align*}
% f(x^{(k)}+d) = f^{(k)} + (g^{(k)})^T d + \frac{1}{2} d^T \nabla^2 f(y) d.
% \end{align*}
% Using the definition of actual and predicted reduction and adding and subtracting $\frac{1}{2}d^T W_1^{(k)}d$ yields
% \begin{align*}
% \Delta f^{(k)}  =\Delta m_f^{(k)} + \frac{1}{2} d^T \left( W_1^{(k)} - \nabla^2 f(y) \right) d.
% \end{align*}
% Inequality \eqref{eq:ared_pred_relation} follows from $\|d\|_2^2 \leq n\|d\|_{\infty}^2 \leq n \Delta^2$ and
% \begin{align*}
%  \frac{1}{2} d^T\left(W_1^{(k)}-\nabla^2 f(y)\right) d \geq  -\frac{1}{2} \left\vert d^T\left(W_1^{(k)}-\nabla^2 f(y)\right) d\right\vert \geq -\frac{1}{2} \Vert W_1^{(k)}-\nabla^2 f(y) \Vert  \Vert d \Vert^2_2 \geq - n \Delta^2 M.
% \end{align*}
For $i \in \{1, \ldots, m\}$, there exists $z$ on the line segment from $x^{(k)}$ to $x^{(k)}+d$ such that
\begin{align*}
c_i \left(x^{(k)} + d\right) = c_i^{(k)} + (\nabla c_i^{(k)})^T d+\frac{1}{2} d^T \nabla^2 c_i(z) d = \frac{1}{2} d^T \nabla^2 c_i(z) d
\end{align*}
by feasibility of $d$. Taking the absolute value and using the estimates similar to \cite{fletcher2002a}, we get  \eqref{eq:constraint_bounds}.
% \todo[inline,color=red!10]{Charlie: \\
% For line-search methods, we have for a given step size $\alpha$:
% \begin{align*}
% c_i \left( x^{(k)} + \alpha d \right) & = c_i^{(k)} + \alpha (\nabla c_i^{(k)})^T d + \frac{\alpha^2}{2} d^T \nabla^2 c_i(z) d \\
%     & = c_i^{(k)} - \alpha c_i^{(k)} + \frac{\alpha^2}{2} d^T \nabla^2 c_i(z) d = (1 - \alpha) c_i^{(k)} + \frac{\alpha^2}{2} d^T \nabla^2 c_i(z) d
% \end{align*}
% Therefore 
% \begin{align*}
% |c_i \left( x^{(k)} + \alpha d \right) | \le | (1 - \alpha) c_i^{(k)}| + |\frac{\alpha^2}{2} d^T \nabla^2 c_i(z) d| \le (1 - \alpha)|c_i^{(k)}| + \frac{1}{2} \alpha^2 M \| d \|_2^2
% \end{align*}
% \textcolor{red}{Sven: The RHS has $(1-\alpha)|c_i^{(k)}|$, which is bounded away from zero, and does not converge to zero, unless $\alpha=1$. As we discussed: the line-search step does not satisfy the linearized constraints, unless we search in the null-space ... }
% }
\end{proof}

\begin{lemma}\label{lemma:sufficient_reduction}
Let standard assumptions hold. If $d$ solves \ref{eq:TR-QP}, $x^{(k)}+d$ is acceptable to the funnel if $\TRradius^2 \leq 2 \beta \tau^{(k)} /(m n M)$.
\end{lemma}
\begin{proof}
It holds that
\begin{align*}
h(x^{(k)}+d) = \Vert c(x^{(k)}+d) \Vert_1 = \sum_{i=1}^{m} \vert c_i(x^{(k)}+d)\vert \leq \sum_{i=1}^{m} \frac{1}{2}n \TRradius^2 M = \frac{1}{2} n m \TRradius^2 M.
\end{align*}
If $\TRradius^2 \leq 2 \beta \tau^{(k)} / (nmM)$, then $ h(x^{(k)}+d) \le \beta \tau^{(k)}$; in other words, the step is acceptable to the funnel.%\todo{We actually prove the funnel sufficient condition}
\end{proof}

We note that Lemma~\ref{lemma:sufficient_reduction} actually proves the funnel sufficient condition.

\begin{lemma}
\label{lemma:sufficient_reduction_at_feasible_point}
Let standard assumptions hold, and let $x^{\circ} \in X$ be a feasible point of problem \eqref{eq:NCO} at which MFCQ holds but which is not a KKT point. Then, there exist a neighborhood $\mathcal{N}^{\circ}$ of $x^{\circ}$ and positive constants $\varepsilon, \mu$, and $\kappa$ such that for all $x \in \mathcal{N}^{\circ} \cap X$ and all $\TRradius$ for which
\begin{align*}
\mu h(x) \leq \TRradius \leq \kappa
\end{align*}
it follows that \ref{eq:TR-QP} has a feasible solution $d$ at which the predicted reduction \eqref{eq:pred}
%\todo{We need $\sigma h(x^{(k)})$ here.}
satisfies
\begin{align}
\label{eq:lower_bound_pred}
\Delta m_f \geq \frac{1}{3} \TRradius \varepsilon,
\end{align}
the sufficient reduction condition \eqref{eq:sufficient_reduction} holds, and the actual reduction \eqref{eq:ared} satisfies
\begin{align}
\Delta f \geq \TRradius \sigma h(x+d). 
\end{align}
%\todo{Should this be $\rho$ or $\Delta$ or $1$?}
\end{lemma}
\begin{proof}
The proof follows directly from the proof of \cite[Lemma 4]{fletcher2002a}.
\end{proof}

\begin{lemma}
\label{lemma:finite_inner_loop}
Let standard assumptions hold. Then the inner iterations terminate finitely.
\end{lemma}
\begin{proof}
The active set at a feasible point $x$ is the set of all bound constraints that hold with equality:
\begin{align*}
\mathcal{A}(x) \equaldef \{ i \in \{1, \ldots, n \} ~|~ x_i = 0 \}.
\end{align*}

We denote by $d$ the global solution of \ref{eq:TR-QP}.
If $x^{(k)}$ is a KKT point of \eqref{eq:NCO}, $d = 0$, and the inner loop terminates.
%Otherwise, we assume that the inner loop does not terminate finitely. Then $\Delta\to0$.
Otherwise, we consider two cases:
\begin{enumerate}
\item $h^{(k)} > 0$: There exists an $i \in \{1, \ldots, m\}$ such that, without loss of generality, $c_i^{(k)} > 0$. For all $d$ such that $\Vert d \Vert_{\infty}\leq \TRradius$, it holds that
\begin{align*}
c_i^{(k)} + (\nabla c_i^{(k)})^T d \geq c_i^{(k)} - \TRradius \| \nabla c^{(k)}_i \|_1.
\end{align*}
If either $\| \nabla c_i^{(k)} \| = 0$ or $\TRradius < |c_i^{(k)}| / \| \nabla c_i^{(k)} \|_1$, we have $c_i^{(k)} - \TRradius \| \nabla c^{(k)}_i \|_1 > 0$.
Thus for $\TRradius$ sufficiently small, \ref{eq:TR-QP} is infeasible, and the inner loop terminates finitely.

\item $h^{(k)}=0$: Inactive bound constraints will stay inactive for any direction $\Vert d \Vert_{\infty}\leq\TRradius$ for sufficiently small $\TRradius$. Therefore, we only consider equality and active inequality constraints. Since $x^{(k)}$ is not a KKT point, there exists $s$ with $\Vert s \Vert = 1$ and $\eta > 0$ such that 
\begin{align*}
s^T \nabla f^{(k)} = -\eta < 0, \quad s^T \nabla c_i^{(k)} = 0 \quad \mathrm{for} \quad i \in \{1, \ldots, m\}, \quad s_i\geq0 \quad\mathrm{for} \quad i \in \mathcal{A}^{(k)}.
\end{align*}
We consider the QP-feasible line segment $\alpha \TRradius s$, $\alpha \in [0, 1]$. We construct the function ${\phi(\alpha) = m_f^{(k)}(\alpha \TRradius s)}$ with the properties
\begin{align*}
\phi'(0) = -\TRradius \eta, \quad \phi'' = \TRradius^2 s^T W_1^{(k)} s \leq \TRradius^2 M.
\end{align*}
Hence, if $\TRradius\leq\eta/M$, then $\phi'(0) + \phi'' \leq 0$. From Lemma \ref{lemma:1Dminimization}, it follows that $\phi(0) -\phi(1)\geq \frac{1}{2} \TRradius \eta$. Because $d$ is globally optimal for \ref{eq:TR-QP}, it holds that
\begin{align}
\label{eq:bound_on_pred}
\Delta m_f^{(k)}(d) \ge \Delta m_f^{(k)}(\TRradius s) = m_f^{(k)}(0) - m_f^{(k)}(\TRradius s) = \phi(0) - \phi(1) \ge \frac{1}{2}\TRradius\eta > 0.
\end{align}
If we choose $\TRradius \le \displaystyle \frac{(1-\sigma) \eta}{2nM}$, then, combining \eqref{eq:ared_pred_relation} and \eqref{eq:bound_on_pred}, we have
\begin{align*}
\Delta f^{(k)} & \overset{\eqref{eq:ared_pred_relation}}{\ge} \Delta m_f^{(k)} - n \TRradius^{(l)} \TRradius^{(l)} M \\
                & \ge \Delta m_f^{(k)} - n \frac{(1-\sigma) \eta}{2nM} \TRradius^{(l)} M = \Delta m_f^{(k)} - (1-\sigma) \frac{1}{2} \TRradius^{(l)} \eta \\
                & \overset{\eqref{eq:bound_on_pred}}{\ge} \Delta m_f^{(k)} - (1-\sigma) \Delta m_f^{(k)}(d) = \sigma \Delta m_f^{(k)} \\
                & > 0 = h^{(k)}.
\end{align*}
These are the necessary conditions of an $f$-type iteration. If $\TRradius^2 \le \displaystyle \frac{2 \beta \tau^{(k)}}{mnM}$, $x^{(k)} + d$ is acceptable to the funnel. Thus, if $\TRradius$ is chosen small enough, an $f$-type step is taken.
\end{enumerate}
\end{proof}

\begin{remark}
We note that the condition that $d$ be the global solution of \ref{eq:TR-QP} holds if the Hessian is positive definite on the null space of the linearized equality constraints. Otherwise, it can be replaced by a Cauchy-point construction that explicitly produces a unit step, $s$ in \eqref{eq:bound_on_pred}, and then requires that \ref{eq:TR-QP} compute a step that is at least as good as the Cauchy prediction.
\end{remark}

\begin{theorem}
If standard assumptions hold, the funnel SQP algorithm has one of the following outcomes:
\begin{enumerate}
\item The restoration phase fails to find a point $x$ that is both acceptable to the funnel and for which \ref{eq:TR-QP} has a feasible direction for some $\TRradius \geq \TRradius^{\circ}$. In this case the restoration phase converges to an infeasible limit point.
\item A KKT point of problem \eqref{eq:NCO} is found ($d = 0$ solves \ref{eq:TR-QP} for some $k$).
\item There exists an accumulation point that is feasible and either is a KKT point or fails to satisfy MFCQ.
\end{enumerate}
\end{theorem}

\begin{proof}
It is enough to consider case 3. As proven in Lemma \ref{lemma:finite_inner_loop}, the inner loop is finite, and therefore the trust-region funnel SQP method produces an infinite sequence of iterates. All iterates lie in the compact set $X$; therefore, there exists a converging subsequence. We consider two cases:
\begin{enumerate}
\item The main sequence contains an infinite number of $h$-type steps: in this case we pick a subsequence containing purely $h$-type steps. For $h$-type iterations, we know $h^{(k)}\to 0$ and $\tau^{(k)}\to 0$. In particular, $\tau^{(k+1)} < \tau^{(k)}$. Therefore, there exists a converging subsequence whose index set is denoted by $\mathcal{S}$ and whose limit point is denoted by $x^{\infty}$. For $k\in\mathcal{S}$, it holds that
\begin{align*}
x^{(k)}\to x^{\infty},\quad h^{(k)}\to 0,\quad \tau^{(k+1)} < \tau^{(k)}.
\end{align*}
Thus, $x^{\infty}$ must be feasible. We take another case distinction.
    \begin{enumerate}
    \item MFCQ is not satisfied at $x^{\infty}$; then the claim holds.
    \item MFCQ is satisfied at $x^{\infty}$. For a proof by contradiction, we assume that $x^{\infty}$ is not a KKT point. The vectors $\nabla c_i^{\infty}$ for $i \in \{1, \ldots, m\}$ are linearly independent, and the MFCQ set is not empty. For sufficiently large $k\in\mathcal{S}$, $x^{(k)}$ lies in the neighborhood $\mathcal{N}^{\infty}$. If $QP(x^{(k)},\TRradius)$ has a feasible solution and 
    \begin{align*}
    \mu h^{(k)} \leq \TRradius \leq \min \left\{ \sqrt{\frac{2\beta h_{\max}^{(k)}}{mnM}},\kappa\right\},
    \end{align*}
    the algorithm performs an $f$-type iteration. For $k$ sufficiently large, we get
    \begin{align*}
    \mu h^{(k)} \leq \TRradius \leq \sqrt{\frac{2\beta h_{\max}^{(k)}}{mnM}}.
    \end{align*}
    We see that the upper bound is more than twice the lower bound. Given the update rule, we find a $\TRradius$ that lies in the interval. Therefore an $f$-type step occurs, which is a contradiction; that is,  $x^{\infty}$ is a KKT point.
    \end{enumerate}
\item The main sequence contains finitely many $h$-type steps. Therefore, there exists an index $\overline{k}$ such that for all $k>\overline{k}$ the sequence $\{f^{(k)}\}$ is strictly monotonically decreasing. From Theorem \ref{theorem:convergence_to_feasibility} we know $h^{(k)}\to0$ for $k \ge \overline{k}$. We deduce that any accumulation point $x^{\infty}$ is feasible. Because $f(x)$ is bounded on $X$ and $\{f^{(k)}\}$ is converging and, in particular, it is a Cauchy sequence, it follows that $\sum_{k\geq\overline{k}}\Delta f^{(k)}$ is converging. Consider again two cases:
    \begin{enumerate}
    \item MFCQ is not satisfied; then there is nothing to show.
    \item MFCQ is satisfied, and we proceed again by deriving a contradiction. A sufficient condition for accepting an $f$-type step is that $\TRradius$ lies in the following interval:
    \begin{align*}
        \mu h^{(k)} \leq \TRradius \leq \min \left\{ \sqrt{\frac{2\beta \tau^{(\overline{k})}}{mnM}},\kappa\right\}.
    \end{align*}
    The funnel parameter $\tau^{(\overline{k})}$ on the right side is constant; therefore the right-hand side is constant, and we denote it by $\overline{\TRradius}>0$. For sufficiently large $k$, we can guarantee that the right side is greater than twice the lower bound. The inner loop decreases $\TRradius$ such that  either it falls into the interval and will be accepted as an {$f$-type} step or it is already accepted beforehand. This guarantees that a trust-region radius ${\TRradius^{(k)} \ge \min\{\frac{1}{2} \overline{\TRradius},\TRradius^{\circ}\}}$ is picked. We deduce from \eqref{eq:lower_bound_pred} and \eqref{eq:sufficient_reduction} that \\ ${\Delta f^{(k)}\geq\frac{1}{3}\sigma \varepsilon \min\{\frac{1}{2}\overline{\TRradius},\TRradius^{\circ}\}}$, which is a contradiction to the convergence of $\sum_{k\geq\bar{k}}\Delta f^{(k)}$. Hence, $x^{\infty}$ must be a KKT point.
    \end{enumerate}
\end{enumerate}
\end{proof}

% \subsection{Global Convergence of Funnel Line-Search Method}
% The proof in this section is adapted from the proof in \cite{waechter2005}.

% \section{An Attempt to Unify the Global Convergence Proof}
% \input{05a-ls_tr_comparison_convergence_proof}

\section{Simulation Results}
\label{sec:simulation_results}

The funnel method has been implemented in \texttt{Uno} as a globalization strategy and is available at \url{https://github.com/david0oo/Uno/tree/funnel_method}.
We compare four different algorithmic configurations on a subset of 278 small instances of the \texttt{CUTEst} test set~\cite{gould2014}. All four versions employ the restoration SQP method but differ in their globalization strategy (filter or funnel) and their globalization mechanism (line-search or trust-region method).

We have excluded unconstrained problems from the test set because filter and funnel methods behave identically on these problems (the switching condition \eqref{eq:switching_condition} is always satisfied, and the same Armijo condition is enforced). 
In our numerical experiments we compare the funnel method with the default filter methods of \texttt{Uno}: we start from a bird's-eye position and move to a successively more detailed analysis. We close this section with some test results on the Maratos effect.

\subsection{Implementation Details}

%First, a $\zeta > 0$ is chosen and an initial adding term $\xi^0$ is determined. Let $w_{ii}$ denote the diagonal element of $W$. If $\min_i w_{ii}>0$, then $\xi^0\gets 0$ otherwise $\xi^0\gets \min_i w_{ii}$.
% In case, there are negative eigenvalues, the regularization term is increased, i.e., $\xi^{k+1}\gets \max\{2\xi^{k}, \zeta\}$. Otherwise, the modified Hessian is accepted. This procedure is repeated until a positive definite Hessian matrix is found.

The algorithmic parameters of the funnel method and the filter method are chosen as in Table \ref{tab:implementation_parameters}. Both methods are initialized with the same parameters to guarantee that they start in similar conditions.
Following the discussion of the funnel update strategy in Section \ref{sec:funnel_mechanism}, we found that  balancing the update of the funnel width  is important; in other words, the funnel reduction should be neither too strict nor too slow. Therefore, we picked $\kappa = 0.5$, which worked well for our simulations.
The filter has a maximum capacity of $50$ entries. For both methods, the initial upper bound on the constraint violation is set according to Equation~\eqref{eq:funnel_initialization}.
The maximum number of outer iterations is set to $4,000$.

\begin{table}[htbp!]
\centering
\caption{Parameter values of the funnel method and the filter method.}
\begin{tabular}{|l|cl|cl|cl|}
\hline 
& Parameter & Value & Parameter & Value & Parameter & Value\\
\hline
\multirow{ 2}{*}{Funnel method} & $\bar{\tau}$ & $100$ & $\bar{\kappa}$ & $1.25$ & $\kappa$ & $0.5$ \\
& $\delta$ & $0.999$ & $\sigma$ & $10^{-4}$ & $\beta$ & $0.99$\\
\hline
Filter method & $\bar{\tau}$ & 100 & $\bar{\kappa}$ & 1.25 & $\beta$ & 0.999 \\
\hline
\end{tabular}
\label{tab:implementation_parameters}
\end{table}

% This leads to the only difference that the filter method adds entries to the filter and the funnel method gradually decreases the funnel.

The QP solver available in \texttt{Uno} is \texttt{bqpd}~\cite{fletcher2000, fletcher1995}, a reliable null-space method for indefinite quadratic optimization.
The trust-region method uses the exact Hessian matrix.
A sufficient condition for the well-posedness of the line-search subproblems is ensured by iteratively adding a multiple of the identity to $W^{(k)}$ until the resulting matrix is positive definite~\cite[p.~51]{Nocedal2006}. The inertia of the matrix is computed by \texttt{MA57}~\cite{HSL}.

\texttt{Uno} terminates at $(x^*, \constraintmultipliers^*, \boundmultipliers^*)$ with one of the following outcomes:
\begin{itemize}
\item {\bf KKT point found}, if the final iterate satisfies an approximate KKT condition:
\begin{equation*}
\|\nabla \mathcal{L}(x^*, 1, \constraintmultipliers^*, \boundmultipliers^*) \| \le \varepsilon,
\quad
\|c(x^*)\| \le \varepsilon,
\quad
\| x^* \odot \boundmultipliers^* \| \le \varepsilon,
\end{equation*}
where the tolerance $\varepsilon$ is set to $10^{-6}$ for our numerical results.
\item {\bf Infeasible stationary point found}, if an approximate KKT condition of the feasibility problem holds:
\begin{equation*}
\left\|
\begin{pmatrix}
\nabla \mathcal{L}(x^*, 0, \constraintmultipliers^*, \boundmultipliers^*) \\
e + \constraintmultipliers^* - \boundmultipliers_u^* \\
e - \constraintmultipliers^* - \boundmultipliers_v^*
\end{pmatrix} \right\| \le \varepsilon,
\quad
\|c(x^*)\| > \varepsilon,
\quad
\left\|
\begin{pmatrix}
x^* \odot \boundmultipliers^* \\
u^* \odot \boundmultipliers_u^* \\
v^* \odot \boundmultipliers_v^*
\end{pmatrix}
\right\| \le \varepsilon,
\end{equation*}
where $u, v \ge 0$ are the elastic variables introduced in \ref{eq:FQP} and $\boundmultipliers_u, \boundmultipliers_v \ge 0$ are the corresponding bound multipliers.
\item {\bf Maximum number of iterations}, if we reached the iteration bound.
\item {\bf Unbounded solution}, if $f(x^*) < -10^{20}$ for an $\varepsilon$-feasible point ($\| c(x^*) \| \le \varepsilon$).
\item {\bf Small feasible step}, if the trust-region radius $\TRradius \le 10^{-16}$ and $\|c(x^*)\| \le \varepsilon$. This indicates that the problem may be ill-conditioned or may not be differentiable.
\end{itemize}

We consider a problem to be solved correctly if both methods found a KKT point or converged to an infeasible stationary point. If one of the methods converged to an infeasible point while the other found a KKT point, the first method failed to solve the problem. Other reasons for failure include the step size becoming too small or an excess of maximum number of iterations.
%Since the implementations of funnel and filter globalization strategy are the same except for one file and due to
% Because of the similarities between both approaches, the behavior of both strategies in \texttt{Uno} is very similar. 
The simulations were carried out on an Intel Core i7-10810U CPU, and the corresponding log files are available at
\url{https://github.com/david0oo/uno_funnel_results}.

\subsection{Comparison between Funnel and Filter}

%\todo[inline,color=blue!20]{For David: We should only use Hessian or gradient evaluations, because a line-search method always looses out in terms of function evaluations, because the back-tracking tends to use more function evaluations. \\
%I think we only need one plot here.}

Detailed results of this comparison are tabulated in Appendices \ref{App:TR} and \ref{App:LS}.
Figure~\ref{fig:comparison_all_solvers_on_all_problems} shows a comparison of the four algorithmic configurations as a Dolan--Mor\'e performance profile~\cite{Dolan2001} for the number of  constraint evaluations.
% On the left, we compare number of Hessian evaluations,
% % \todo{note: this is not equivalent to the number of outer iterations, because the Hessian is reevaluated when we switch phases}
% on the right, we see the number of objective evaluations.
A point $(\alpha, \beta)$ on the graph means that a method solves the fraction $\beta$ of all test problems within $\alpha$ times the number of evaluations of the virtual best solver---the solver that performs best for every instance~\cite{gould2015}. The higher and the more to the left, the better.
The performance profiles demonstrate that the funnel method slightly outperforms the filter method in terms of Hessian and constraint evaluations.
For these two metrics, line-search methods perform much worse than trust-region methods. The cause seems to stem from the convexification procedure: the trust-region methods solve indefinite QPs and exploit the negative curvature of the problems. If the trust-region QPs are convexified as are the line-search QPs, the performance similarly degrades.

%\todo[inline]{Feel free to remove figure and discussion about constraint evaluations.}

\begin{figure}[htbp!]
\centering
% \def\scaling{1.4}
% Upper Left Figure
% \begin{subfigure}{.47\linewidth}
% \centering
% \includegraphics[width=1\textwidth]{graphics/performance_profile_all_solvers_n_eval_hessian.pdf}
% \caption{Hessian evaluations.}
% \end{subfigure}
% %\hfill
% \begin{subfigure}{.47\linewidth}
% \centering
\includegraphics[width=0.5\textwidth]{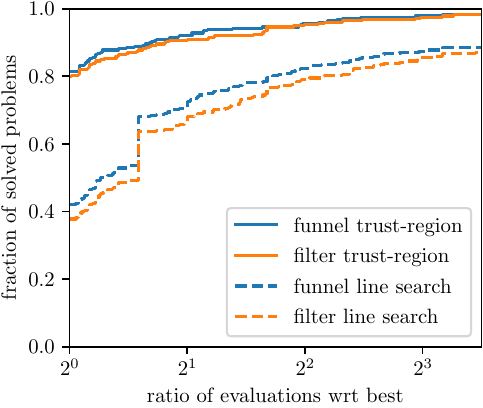}
% \caption{Constraint evaluations.}
% \end{subfigure}
\caption{Performance profiles for all four algorithmic configurations with respect to constraint evaluations.}
\label{fig:comparison_all_solvers_on_all_problems}
\end{figure}

Figure~\ref{fig:distribution} presents a distribution plot of the filter and funnel methods with trust region (left) and line search (right) on the \texttt{CUTEst} instances, using the number of constraint evaluations as the metric. Each point corresponds to one problem instance. Points on the diagonal correspond to instances where the funnel method and the filter method took the same number of iterations, points above the diagonal correspond to instances where the funnel method was faster than the filter method, and points below the diagonal correspond to instances where the filter method was faster. We observe that on many instances, the two methods behave almost identically. For the line-search methods, we observe that the funnel method has more wins (points above the diagonal) than the filter method, which is also observed in the performance profiles.

\begin{figure}[htbp!]
\centering
\def\scaling{1.4}
% Upper Left Figure
\begin{subfigure}{.49\linewidth}
\centering
\includegraphics[width=1\textwidth]{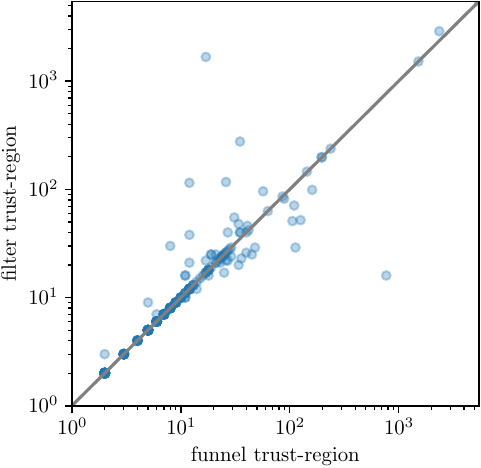}
\caption{Trust-region methods.}
\end{subfigure}
%\hfill
\begin{subfigure}{.49\linewidth}
\centering
\includegraphics[width=1\textwidth]{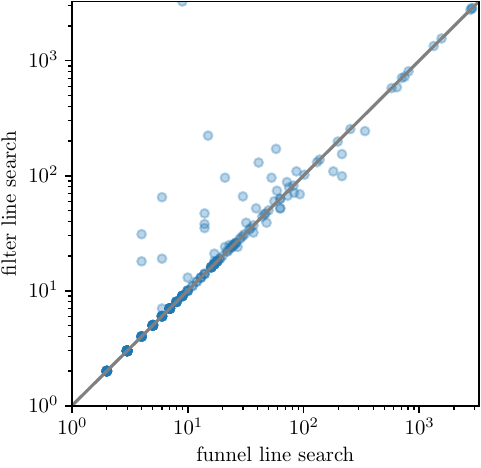}
\caption{Line-search methods.}
\end{subfigure}
\caption{Distribution plot of constraint evaluations for funnel vs filter globalization strategies. }
\label{fig:distribution}
\end{figure}

Next, we compared the number of iterations and the number of evaluations of all functions and derivatives. This more detailed analysis reveals that both trust-region methods take different iteration paths for only 49 instances. The results of the line-search methods also  differ for only 49 of the problems. 
In Figures~\ref{fig:f_and_h_and_r_steps_trust_region} and \ref{fig:f_and_h_and_r_steps_line_search} we plot for each of these instances a vertical bar that shows the number of $f$-type (blue), $h$-type (orange), and restoration steps (green) for the trust-region and line-search methods, respectively. A missing bar means that the algorithm failed to solve the problem. Overall, the ratio of the three iteration types seems to be similar, except for a few outliers.
% Charlie: I would not list these instances. The viewer can simply have a look!
% For the trust-region version, problems that have quite some different behavior include \texttt{csfi2}, \texttt{gottfr}, \texttt{hs027}, \texttt{hs081}, \texttt{hs103}, \texttt{hs106}, \texttt{minmaxbd}, \texttt{polak2}, \texttt{powellbs}, \texttt{rk23}.
We  observe that fewer problems can be solved by line-search methods than by trust-region methods, as is also shown in the performance profiles.
% Excluding problems where one of the line search solvers fails, problems with quite different behavior include \texttt{aljazzaf}, \texttt{hs027}, \texttt{hs102}, \texttt{himmelbd}.
% We also observe that the line-search variant has more problems that only perform h-type steps, indicating that line-search has a harder time finding feasible points \todo{David: Is this the correct implication? not sure.}.
% Charlie: I think the convexification sort of destroys the meaning of the original objective. Therefore, we end up working on decreasing infeasibility most of the time.

\begin{figure}[h!]
\includegraphics[width=1\textwidth]{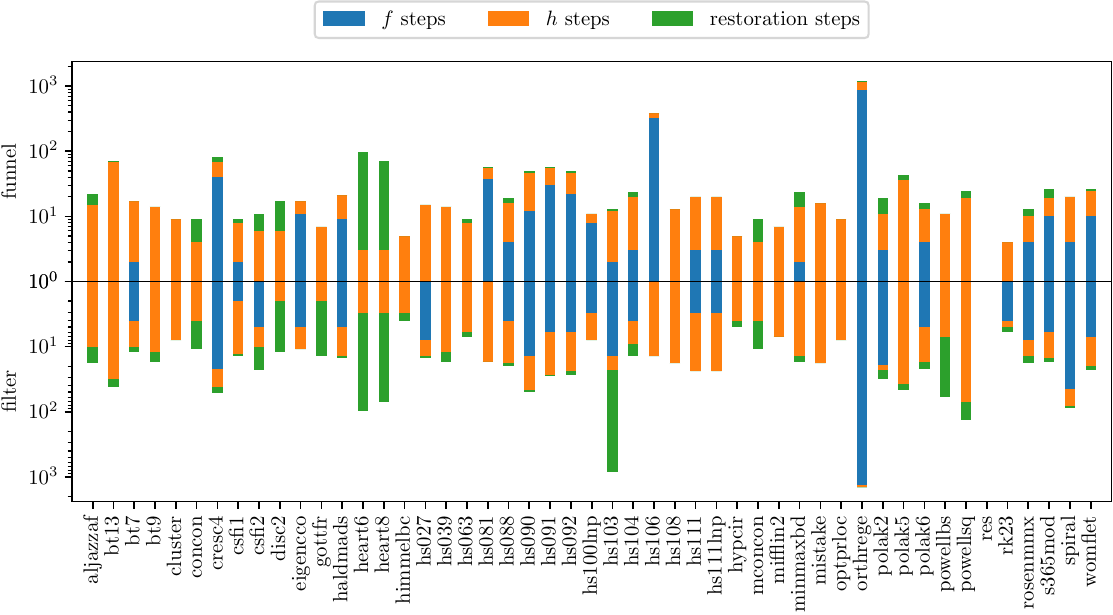}
\caption{Comparison of the number of $f$-type (blue), $h$-type (orange), and restoration steps (green) for the trust-region funnel method and the trust-region filter method.
}
\label{fig:f_and_h_and_r_steps_trust_region}
\end{figure}

\begin{figure}[h!]
\includegraphics[width=1\textwidth]{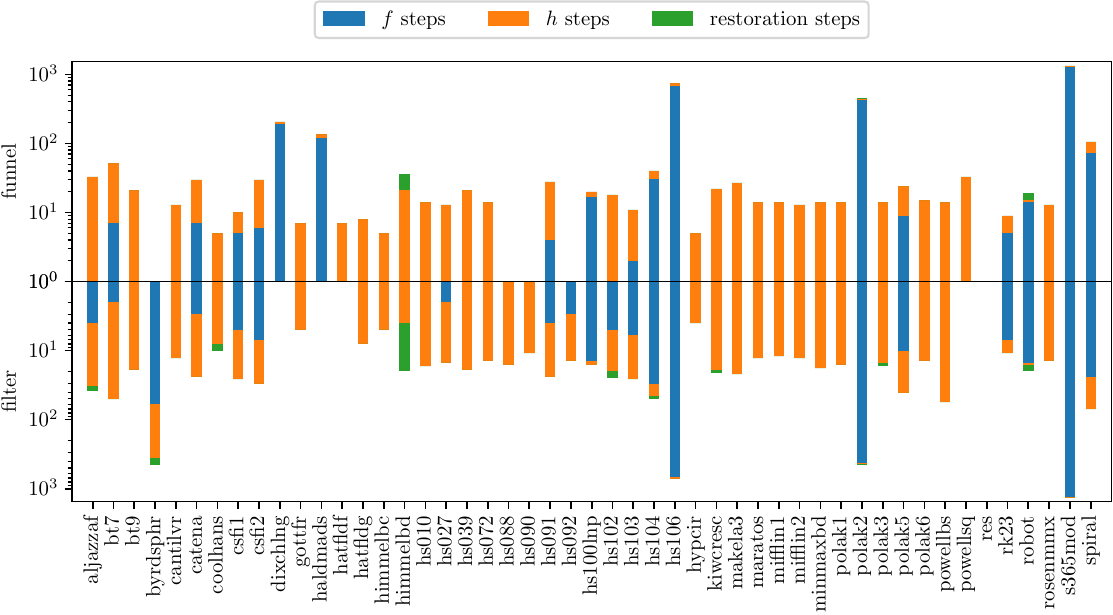}
\caption{Comparison of the number of $f$-type (blue), $h$-type (orange), and restoration steps (green) for the line-search funnel method and the line-search filter method.
}
\label{fig:f_and_h_and_r_steps_line_search}
\end{figure}

\subsection{Insights from Illustrative Examples}

We finish our discussion of the numerical results by zooming in to two illustrative examples that highlight the similarities and differences between filter and funnel methods.

\subsubsection{\texttt{powellbs} Example}

The \texttt{powellbs} instance for the trust-region methods is given by
\begin{align*}
\min_{x \in \R^2} \quad & 0 \\
\mathrm{s.t.} \quad & -1 + 10000 x_1 x_2 = 0, \\
                    & -1.0001 + e^{-x_1} + e^{-x_2} = 0,
\end{align*}
with initial point $x^{(0)} = (0, 1)^T$. The optimal solution is approximately $(\num{1.1e-05}, 9.1)^T$.
Figure~\ref{fig:contraction_infeasibility_powellbs} shows the infeasibility and the funnel width (in log scale) as functions of the outer iterations for the trust-region funnel and filter methods. The funnel width decreases linearly, because only $h$-type steps are taken. Interestingly, both methods are identical until iteration 5. From iteration 6 onwards, the funnel allows more non-monotonicity, and the method enters the regime of quadratic convergence, while the filter blocks progress.  This causes infeasibility of the QP and triggers the switch to feasibility restoration. Many iterations are required to find a filter-acceptable point, after which the convergence is quadratic. 

\begin{figure}[h!]
\centering
\includegraphics[width=0.7\textwidth]{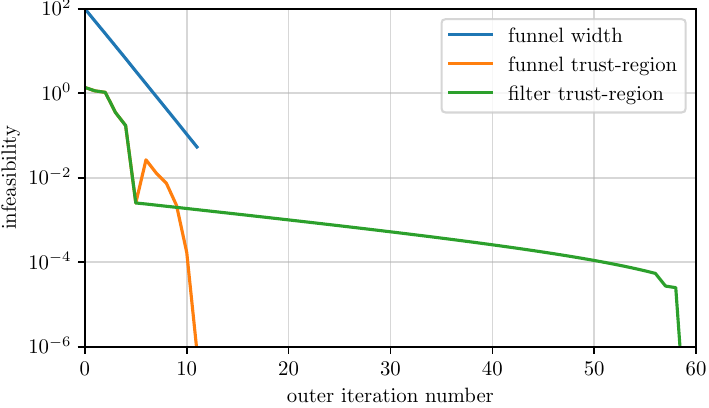}
\caption{Infeasibility and funnel width as functions of the outer iterations for trust-region funnel and filter algorithms on the  \texttt{powellbs} problem.
}
\label{fig:contraction_infeasibility_powellbs}
\end{figure}

\subsubsection{\texttt{maratos} Example}

Filter (or funnel) methods do not fail for the standard \texttt{maratos} and can be shown to take unit steps arbitarily close to a feasible point. However, in ~\cite{fletcher2006} another example was introduced proving that, in general, filter methods (including \texttt{filterSQP}) do suffer from the Maratos effect:
\begin{align*}
\min_{x \in \R^2} & \quad 2\left(x_1^2+x_2^2-1\right)-x_1 \\
\mathrm{s.t.} &\quad x_1^2+x_1^2-1=0 .
\end{align*}
If the starting point is chosen as $x = (\cos(t), \sin(t))$ for $t > 0$ small and $\lambda = \frac{3}{2}$ ($-\frac{3}{2}$ with our notation), the quadratic model of the objective predicts a decrease although the objective increases, which violates the Armijo condition. Thus, the full step is rejected.
% In our simulations, the funnel and filter trust-region methods took the same kind of iterations.

The iterations of the trust-region funnel SQP and line-search funnel SQP methods starting from $(x_1, x_2, \constraintmultipliers) = (\frac{\sqrt{2}}{2}, \frac{\sqrt{2}}{2}, -\frac{3}{2})$ are given in Tables~\ref{tab:maratos_tr_funnel_sqp} and \ref{tab:maratos_ls_funnel_sqp}, respectively.
We see in both cases that in the first iteration, the full step increases the objective and the constraint violation, which yields a step rejection.
This demonstrates that the funnel method also suffers from the Maratos effect. Therefore, appropriate measures such as second-order corrections, a watchdog strategy, or non-monotone techniques are required for fast local convergence.

\begin{table}[h!]
\centering
\renewcommand{\arraystretch}{0.8} % Adjust this value to reduce line spacing
\small
\caption{Iterations of trust-region funnel SQP method on \texttt{maratos} example.}
\begin{tabular}{|c|c|c|c|c|c|c|c|c|}
\hline
$k$ & $l$ & $\TRradius^{(k,l)}$ & $\tau^{(k)}$ & $\| d^{(k, l)} \|_\infty$ & $f^{(k)}$ & $h^{(k)}$ & $\| \nabla \mathcal{L} \|$ & status \\
\hline
0 & -- & $1.00e+01$  & $1.00e+02$ & -- & $-0.707$ & $5.28e-10$ & $7.65e-01$ & \texttt{initial point} \\
1 & 1 & $1.00e+01$  & $1.00e+02$ & $5.00e-01$ & $-0.207$ & $5.00e-01$ & -- & \texttt{rejected (Armijo)} \\
-- & 2 & $2.50e-01$  & $1.00e+02$ & $2.50e-01$ & $-0.707$ & $1.25e-01$ & -- & \texttt{rejected (Armijo)} \\
-- & 3 & $1.25e-01$  & $1.00e+02$ & $1.25e-01$ & $-0.770$ & $3.13e-02$ & $8.59e-01$ & \texttt{$f$-type step} \\
2 & 1 & $2.50e-01$  & $1.00e+02$ & $2.50e-01$ & $-0.814$ & $8.69e-02$ & $4.18e-01$ & \texttt{$f$-type step} \\
3 & 1 & $5.00e-01$  & $1.00e+02$ & $2.71e-01$ & $-0.883$ & $7.60e-02$ & $5.86e-02$ & \texttt{$f$-type step} \\
4 & 1 & $5.00e-01$  & $1.00e+02$ & $6.30e-02$ & $-0.992$ & $5.06e-03$ & $2.86e-03$ & \texttt{$f$-type step} \\
5 & 1 & $5.00e-01$ & $1.00e+02$ & $2.55e-03$ & $-1.000$ & $1.28e-05$ & $1.37e-05$ & \texttt{$f$-type step} \\
6 & 1 & $5.00e-01$  & $1.00e+02$ & $9.77e-06$ & $-1.000$ & $1.37e-10$ & $1.88e-10$ & \texttt{$\varepsilon$-optimal} \\
\hline
\end{tabular}
\label{tab:maratos_tr_funnel_sqp}
\end{table}

\begin{table}[h!]
\caption{Iterations of line-search funnel SQP method on \texttt{maratos} example.}
\centering
\renewcommand{\arraystretch}{0.8} % Adjust this value to reduce line spacing
\small
\begin{tabular}{|c|c|c|c|c|c|c|c|c|c|}
\hline
$k$ & $l$ & $\alpha^{(l)}$ & regulariz. & $\tau^{(k)}$ & $\| d^{(k, l)} \|_\infty$ & $f^{(k)}$ & $h^{(k)}$ & $\| \nabla \mathcal{L} \|$ & status \\
\hline
0 & -- & -- & -- & $1.00e+02$ & -- & -0.707 & $5.28e-10$ & $7.65e-01$ & \texttt{initial point} \\
1 & 1 & 1 & $1.00e-04$ & $1.00e+02$ & $5.00e-01$ & -0.207 & $5.00e-01$ & -- & \texttt{rejected (Armijo)} \\
-- & 2 & 0.5 & - & $1.00e+02$ & $2.50e-01$ & -0.707 & $1.25e-01$ &  $4.31e-01$ & \texttt{$f$-type step} \\
2 & 1 & 1 & $1.00e-04$ & $1.00e+02$ & $4.81e-01$ & -0.605 & $2.58e-01$ & -- & \texttt{rejected (Armijo)} \\
-- & 2 & 0.5 & -- & $1.00e+02$ & $2.40e-01$ & -0.785 & $1.27e-01$ & $2.10e-01$ & \texttt{$f$-type step} \\
3 & 1 & 1 & $1.00e-04$ & $1.00e+02$ & $2.40e-01$ & -0.913 & $5.79e-02$ & $2.30e-02$ & \texttt{$f$-type step} \\
4 & 1 & 1 & $1.00e-04$ & $1.00e+02$ & $2.76e-02$ & -0.998 & $1.35e-03$ &  $9.91e-04$ & \texttt{$f$-type step} \\
5 & 1 & 1 & $1.00e-04$ & $1.00e+02$ & $6.74e-04$ & -1.000 & $8.86e-07$  & $1.24e-06$ & \texttt{$f$-type step} \\
6 & 1 & 1 & $1.00e-04$ & $1.00e+02$ & $8.65e-07$ & -1.000 & $9.44e-13$ &  $9.59e-11$ & \texttt{$\varepsilon$-optimal} \\
\hline
\end{tabular}
\label{tab:maratos_ls_funnel_sqp}
\end{table}

% We leave this out.....
% \subsection{Runtime Comparison}
% In this section, the runtimes of the funnel and filter algorithms are compared. For this purpose, all tested instances of the \texttt{CUTEst} test set where run 100 times and the minimum of the runtimes are plotted in an absolute and in a relative performance plot.

% \begin{figure}[H]
% \begin{subfigure}[b]{0.49\textwidth}
%          \centering
%         \includegraphics[width=1\textwidth]{graphics/default_filter_comparison_absolute_cpu_time.pdf}
%          \caption{Absolute CPU time.}
%          % \label{fig:cutest_hessian}
% \end{subfigure}
%      \hfill
% \begin{subfigure}[b]{0.49\textwidth}
%          \centering
%         \includegraphics[width=1\textwidth]{graphics/default_filter_comparison_relative_cpu_time.pdf}
%          \caption{Relative CPU.}
%          % \label{fig:cutest_cpu}
% \end{subfigure}
% \caption{Performance profiles of filter method and funnel method with default parameters.
% % On the left are the Hessian evaluations and objective function evaluations are on the right.
% }
% \end{figure}

\section{Conclusion}
\label{sec:conclusion}

We consider a generic double-loop framework for solving nonlinearly constrained optimization problems that has been implemented in the open-source solver \texttt{Uno}. To illustrate the framework, we studied several variants of a restoration SQP method for nonlinearly constrained optimization problems, with an emphasis on a new globalization strategy called the funnel method. The method was presented through the lens of filter methods, whose theory served as a starting point for a theoretical analysis. We derived the global convergence for the trust-region funnel SQP method and outlined the main differences and similarities in the different proofs. An implementation of the funnel strategy in the \texttt{Uno} solver proved to slightly outperform its filter counterpart with respect to  constraint evaluations for a subset of \texttt{CUTEst} instances, while being also easier to implement.

\bibliographystyle{plain}
\bibliography{bibs/NLP}

\vfill
{\em The submitted manuscript has been created by UChicago Argonne, LLC, Operator of Argonne National Laboratory ("Argonne”). Argonne, a U.S. Department of Energy Office of Science laboratory, is operated under Contract No. DE-AC02-06CH11357. The U.S. Government retains for itself, and others acting on its behalf, a paid-up nonexclusive, irrevocable worldwide license in said article to reproduce, prepare derivative works, distribute copies to the public, and perform publicly and display publicly, by or on behalf of the Government. The Department of Energy will provide public access to these results of federally sponsored research in accordance with the DOE Public Access Plan (\href{http://energy.gov/downloads/doe-public-access-plan}{\tt http://energy.gov/downloads/doe-public-access-plan}).

\newpage
\appendix

\section{Detailed Results on \texttt{CUTEst} problems: Number of Function and Derivative Evaluations for Trust-Region Algorithms}\label{App:TR}
\footnotesize
\renewcommand{\arraystretch}{1}
\setlength{\tabcolsep}{4pt}
\begin{longtable}{|c||c|ccccc||c|ccccc|}
\hline
 & \multicolumn{6}{c|}{trust-region funnel method}           & \multicolumn{6}{c|}{trust-region filter method} \\
\hline
instance & status & $f$ & $c$ & $\nabla f$ & $\nabla c$ & $\nabla^2 \mathcal{L}$ & status & $f$ & $c$ & $\nabla f$ & $\nabla c$ & $\nabla^2 \mathcal{L}$ \\
\hline
aircrfta & KKT & 4 & 4 & 4 & 4 & 3 & KKT & 4 & 4 & 4 & 4 & 3 \\
airport & KKT & 13 & 13 & 13 & 13 & 12 & KKT & 13 & 13 & 13 & 13 & 12 \\
aljazzaf & KKT & 17 & 26 & 23 & 23 & 23 & KKT & 12 & 22 & 19 & 19 & 19 \\
allinitc & KKT & 29 & 29 & 29 & 29 & 28 & KKT & 29 & 29 & 29 & 29 & 28 \\
alsotame & KKT & 5 & 5 & 5 & 5 & 4 & KKT & 5 & 5 & 5 & 5 & 4 \\
argauss & infeasible & 1 & 3 & 3 & 3 & 3 & infeasible & 1 & 3 & 3 & 3 & 3 \\
avgasa & KKT & 2 & 2 & 2 & 2 & 1 & KKT & 2 & 2 & 2 & 2 & 1 \\
avgasb & KKT & 2 & 2 & 2 & 2 & 1 & KKT & 2 & 2 & 2 & 2 & 1 \\
avion2 & KKT & 23 & 23 & 11 & 11 & 10 & KKT & 23 & 23 & 11 & 11 & 10 \\
batch & KKT & 9 & 9 & 9 & 9 & 8 & KKT & 9 & 9 & 9 & 9 & 8 \\
biggsc4 & KKT & 2 & 2 & 2 & 2 & 1 & KKT & 2 & 2 & 2 & 2 & 1 \\
booth & KKT & 2 & 2 & 2 & 2 & 1 & KKT & 2 & 2 & 2 & 2 & 1 \\
bt1 & infeasible & 2 & 2 & 2 & 2 & 2 & infeasible & 2 & 2 & 2 & 2 & 2 \\
bt10 & KKT & 7 & 7 & 7 & 7 & 6 & KKT & 7 & 7 & 7 & 7 & 6 \\
bt11 & KKT & 7 & 7 & 7 & 7 & 6 & KKT & 7 & 7 & 7 & 7 & 6 \\
bt12 & KKT & 5 & 5 & 5 & 5 & 4 & KKT & 5 & 5 & 5 & 5 & 4 \\
bt13 & KKT & 105 & 110 & 73 & 73 & 75 & KKT & 58 & 71 & 42 & 42 & 46 \\
bt2 & KKT & 13 & 13 & 13 & 13 & 12 & KKT & 13 & 13 & 13 & 13 & 12 \\
bt3 & KKT & 2 & 2 & 2 & 2 & 1 & KKT & 2 & 2 & 2 & 2 & 1 \\
bt4 & KKT & 8 & 8 & 7 & 7 & 6 & KKT & 8 & 8 & 7 & 7 & 6 \\
bt5 & KKT & 9 & 9 & 8 & 8 & 7 & KKT & 9 & 9 & 8 & 8 & 7 \\
bt6 & KKT & 11 & 11 & 10 & 10 & 9 & KKT & 11 & 11 & 10 & 10 & 9 \\
bt7 & KKT & 23 & 23 & 18 & 18 & 17 & KKT & 17 & 21 & 13 & 13 & 14 \\
bt8 & KKT & 12 & 12 & 12 & 12 & 11 & KKT & 12 & 12 & 12 & 12 & 11 \\
bt9 & KKT & 19 & 19 & 15 & 15 & 14 & KKT & 19 & 25 & 18 & 18 & 18 \\
byrdsphr & KKT & 8 & 23 & 14 & 14 & 14 & KKT & 8 & 23 & 14 & 14 & 14 \\
cantilvr & KKT & 13 & 13 & 11 & 11 & 10 & KKT & 13 & 13 & 11 & 11 & 10 \\
catena & KKT & 12 & 12 & 11 & 11 & 10 & KKT & 12 & 12 & 11 & 11 & 10 \\
cb2 & KKT & 7 & 7 & 7 & 7 & 6 & KKT & 7 & 7 & 7 & 7 & 6 \\
cb3 & KKT & 7 & 7 & 7 & 7 & 6 & KKT & 7 & 7 & 7 & 7 & 6 \\
chaconn1 & KKT & 5 & 5 & 5 & 5 & 4 & KKT & 5 & 5 & 5 & 5 & 4 \\
chaconn2 & KKT & 5 & 5 & 5 & 5 & 4 & KKT & 5 & 5 & 5 & 5 & 4 \\
cluster & KKT & 10 & 10 & 10 & 10 & 9 & KKT & 10 & 10 & 9 & 9 & 8 \\
concon & KKT & 5 & 11 & 10 & 10 & 10 & KKT & 7 & 16 & 12 & 12 & 16 \\
congigmz & KKT & 5 & 6 & 6 & 6 & 6 & KKT & 5 & 6 & 6 & 6 & 6 \\
coolhans & KKT & 3 & 3 & 3 & 3 & 2 & KKT & 3 & 3 & 3 & 3 & 2 \\
core1 & KKT & 6 & 238 & 123 & 123 & 123 & KKT & 6 & 238 & 123 & 123 & 123 \\
coshfun & unbounded & 86 & 86 & 75 & 75 & 74 & unbounded & 86 & 86 & 75 & 75 & 74 \\
cresc4 & KKT & 131 & 161 & 84 & 84 & 85 & KKT & 84 & 99 & 53 & 53 & 60 \\
csfi1 & KKT & 11 & 12 & 10 & 10 & 10 & KKT & 20 & 21 & 15 & 15 & 15 \\
csfi2 & KKT & 7 & 12 & 12 & 12 & 12 & KKT & 13 & 38 & 24 & 24 & 28 \\
dallass & KKT & 18 & 18 & 16 & 16 & 15 & KKT & 18 & 18 & 16 & 16 & 15 \\
deconvc & KKT & 24 & 24 & 15 & 15 & 14 & KKT & 24 & 24 & 15 & 15 & 14 \\
degenlpa & KKT & 2 & 2 & 2 & 2 & 1 & KKT & 2 & 2 & 2 & 2 & 1 \\
degenlpb & KKT & 2 & 2 & 2 & 2 & 1 & KKT & 2 & 2 & 2 & 2 & 1 \\
demymalo & KKT & 8 & 8 & 8 & 8 & 7 & KKT & 8 & 8 & 8 & 8 & 7 \\
dipigri & KKT & 12 & 12 & 9 & 9 & 8 & KKT & 12 & 12 & 9 & 9 & 8 \\
disc2 & infeasible & 9 & 25 & 18 & 18 & 18 & infeasible & 5 & 17 & 13 & 13 & 15 \\
discs & error & -- & -- & -- & -- & -- & error & -- & -- & -- & -- & -- \\
dixchlng & KKT & 10 & 10 & 10 & 10 & 9 & KKT & 10 & 10 & 10 & 10 & 9 \\
dnieper & KKT & 4 & 4 & 4 & 4 & 3 & KKT & 4 & 4 & 4 & 4 & 3 \\
dual1 & KKT & 2 & 2 & 2 & 2 & 1 & KKT & 2 & 2 & 2 & 2 & 1 \\
dual2 & KKT & 2 & 2 & 2 & 2 & 1 & KKT & 2 & 2 & 2 & 2 & 1 \\
dual4 & KKT & 2 & 2 & 2 & 2 & 1 & KKT & 2 & 2 & 2 & 2 & 1 \\
dualc1 & KKT & 2 & 2 & 2 & 2 & 1 & KKT & 2 & 2 & 2 & 2 & 1 \\
dualc2 & KKT & 2 & 2 & 2 & 2 & 1 & KKT & 2 & 2 & 2 & 2 & 1 \\
dualc5 & KKT & 2 & 2 & 2 & 2 & 1 & KKT & 2 & 2 & 2 & 2 & 1 \\
dualc8 & KKT & 2 & 2 & 2 & 2 & 1 & KKT & 2 & 2 & 2 & 2 & 1 \\
eigencco & KKT & 34 & 34 & 18 & 18 & 17 & KKT & 20 & 20 & 12 & 12 & 11 \\
eigmaxc & KKT & 5 & 6 & 6 & 6 & 6 & KKT & 5 & 6 & 6 & 6 & 6 \\
eigminc & KKT & 5 & 6 & 6 & 6 & 6 & KKT & 5 & 6 & 6 & 6 & 6 \\
expfita & KKT & 13 & 13 & 13 & 13 & 12 & KKT & 13 & 13 & 13 & 13 & 12 \\
extrasim & KKT & 2 & 2 & 2 & 2 & 1 & KKT & 2 & 2 & 2 & 2 & 1 \\
fccu & KKT & 4 & 4 & 4 & 4 & 3 & KKT & 4 & 4 & 4 & 4 & 3 \\
fletcher & infeasible & 2 & 2 & 2 & 2 & 2 & infeasible & 2 & 2 & 2 & 2 & 2 \\
genhs28 & KKT & 2 & 2 & 2 & 2 & 1 & KKT & 2 & 2 & 2 & 2 & 1 \\
gigomez1 & KKT & 8 & 8 & 8 & 8 & 7 & KKT & 8 & 8 & 8 & 8 & 7 \\
goffin & KKT & 3 & 3 & 3 & 3 & 2 & KKT & 3 & 3 & 3 & 3 & 2 \\
gottfr & KKT & 8 & 8 & 8 & 8 & 7 & KKT & 14 & 30 & 15 & 15 & 35 \\
gridnetg & KKT & 4 & 4 & 4 & 4 & 3 & KKT & 4 & 4 & 4 & 4 & 3 \\
gridneth & KKT & 5 & 5 & 5 & 5 & 4 & KKT & 5 & 5 & 5 & 5 & 4 \\
gridneti & KKT & 5 & 5 & 5 & 5 & 4 & KKT & 5 & 5 & 5 & 5 & 4 \\
haifas & KKT & 12 & 12 & 9 & 9 & 8 & KKT & 12 & 12 & 9 & 9 & 8 \\
haldmads & KKT & 40 & 40 & 22 & 22 & 21 & KKT & 25 & 26 & 16 & 16 & 16 \\
hatfldf & KKT & 8 & 28 & 14 & 14 & 22 & KKT & 8 & 28 & 14 & 14 & 22 \\
hatfldg & KKT & 5 & 18 & 10 & 10 & 10 & KKT & 5 & 18 & 10 & 10 & 10 \\
hatfldh & KKT & 2 & 2 & 2 & 2 & 1 & KKT & 2 & 2 & 2 & 2 & 1 \\
heart6 & KKT & 4 & 197 & 99 & 99 & 99 & KKT & 5 & 199 & 100 & 100 & 102 \\
heart8 & KKT & 4 & 144 & 72 & 72 & 72 & KKT & 5 & 146 & 73 & 73 & 75 \\
himmelba & KKT & 2 & 2 & 2 & 2 & 1 & KKT & 2 & 2 & 2 & 2 & 1 \\
himmelbc & KKT & 6 & 6 & 6 & 6 & 5 & KKT & 5 & 6 & 5 & 5 & 5 \\
himmelbd & infeasible & 4 & 10 & 7 & 7 & 7 & infeasible & 4 & 10 & 7 & 7 & 7 \\
himmelbe & KKT & 2 & 2 & 2 & 2 & 1 & KKT & 2 & 2 & 2 & 2 & 1 \\
himmelbk & KKT & 6 & 6 & 6 & 6 & 5 & KKT & 6 & 6 & 6 & 6 & 5 \\
himmelp2 & KKT & 9 & 9 & 9 & 9 & 8 & KKT & 9 & 9 & 9 & 9 & 8 \\
himmelp3 & KKT & 5 & 5 & 5 & 5 & 4 & KKT & 5 & 5 & 5 & 5 & 4 \\
himmelp4 & KKT & 5 & 5 & 5 & 5 & 4 & KKT & 5 & 5 & 5 & 5 & 4 \\
himmelp5 & KKT & 12 & 12 & 10 & 10 & 9 & KKT & 12 & 12 & 10 & 10 & 9 \\
himmelp6 & KKT & 2 & 2 & 2 & 2 & 1 & KKT & 2 & 2 & 2 & 2 & 1 \\
hong & KKT & 5 & 5 & 5 & 5 & 4 & KKT & 5 & 5 & 5 & 5 & 4 \\
hs006 & KKT & 3 & 3 & 3 & 3 & 2 & KKT & 3 & 3 & 3 & 3 & 2 \\
hs007 & KKT & 12 & 12 & 9 & 9 & 8 & KKT & 12 & 12 & 9 & 9 & 8 \\
hs008 & KKT & 6 & 6 & 6 & 6 & 5 & KKT & 6 & 6 & 6 & 6 & 5 \\
hs009 & KKT & 5 & 5 & 4 & 4 & 3 & KKT & 5 & 5 & 4 & 4 & 3 \\
hs010 & KKT & 10 & 10 & 10 & 10 & 9 & KKT & 10 & 10 & 10 & 10 & 9 \\
hs011 & KKT & 6 & 6 & 6 & 6 & 5 & KKT & 6 & 6 & 6 & 6 & 5 \\
hs012 & KKT & 8 & 8 & 6 & 6 & 5 & KKT & 8 & 8 & 6 & 6 & 5 \\
hs013 & KKT & 25 & 25 & 25 & 25 & 24 & KKT & 25 & 25 & 25 & 25 & 24 \\
hs014 & KKT & 6 & 6 & 6 & 6 & 5 & KKT & 6 & 6 & 6 & 6 & 5 \\
hs015 & KKT & 4 & 7 & 6 & 6 & 6 & KKT & 4 & 7 & 6 & 6 & 6 \\
hs016 & KKT & 5 & 5 & 5 & 5 & 4 & KKT & 5 & 5 & 5 & 5 & 4 \\
hs017 & KKT & 8 & 8 & 7 & 7 & 6 & KKT & 8 & 8 & 7 & 7 & 6 \\
hs018 & KKT & 7 & 7 & 7 & 7 & 6 & KKT & 7 & 7 & 7 & 7 & 6 \\
hs019 & KKT & 7 & 7 & 7 & 7 & 6 & KKT & 7 & 7 & 7 & 7 & 6 \\
hs020 & KKT & 5 & 5 & 5 & 5 & 4 & KKT & 5 & 5 & 5 & 5 & 4 \\
hs021 & KKT & 2 & 2 & 2 & 2 & 1 & KKT & 2 & 2 & 2 & 2 & 1 \\
hs022 & KKT & 2 & 2 & 2 & 2 & 1 & KKT & 2 & 2 & 2 & 2 & 1 \\
hs023 & KKT & 7 & 7 & 7 & 7 & 6 & KKT & 7 & 7 & 7 & 7 & 6 \\
hs024 & KKT & 3 & 3 & 3 & 3 & 2 & KKT & 3 & 3 & 3 & 3 & 2 \\
hs026 & KKT & 18 & 18 & 18 & 18 & 17 & KKT & 18 & 18 & 18 & 18 & 17 \\
hs027 & KKT & 17 & 17 & 16 & 16 & 15 & KKT & 21 & 22 & 16 & 16 & 16 \\
hs028 & KKT & 2 & 2 & 2 & 2 & 1 & KKT & 2 & 2 & 2 & 2 & 1 \\
hs029 & KKT & 8 & 8 & 7 & 7 & 6 & KKT & 8 & 8 & 7 & 7 & 6 \\
hs030 & KKT & 2 & 2 & 2 & 2 & 1 & KKT & 2 & 2 & 2 & 2 & 1 \\
hs031 & KKT & 6 & 6 & 6 & 6 & 5 & KKT & 6 & 6 & 6 & 6 & 5 \\
hs032 & KKT & 2 & 2 & 2 & 2 & 1 & KKT & 2 & 2 & 2 & 2 & 1 \\
hs033 & KKT & 5 & 5 & 5 & 5 & 4 & KKT & 5 & 5 & 5 & 5 & 4 \\
hs034 & KKT & 8 & 8 & 8 & 8 & 7 & KKT & 8 & 8 & 8 & 8 & 7 \\
hs035 & KKT & 2 & 2 & 2 & 2 & 1 & KKT & 2 & 2 & 2 & 2 & 1 \\
hs036 & KKT & 3 & 3 & 3 & 3 & 2 & KKT & 3 & 3 & 3 & 3 & 2 \\
hs037 & KKT & 6 & 6 & 6 & 6 & 5 & KKT & 6 & 6 & 6 & 6 & 5 \\
hs039 & KKT & 19 & 19 & 15 & 15 & 14 & KKT & 19 & 25 & 18 & 18 & 18 \\
hs040 & KKT & 5 & 5 & 5 & 5 & 4 & KKT & 5 & 5 & 5 & 5 & 4 \\
hs041 & KKT & 2 & 2 & 2 & 2 & 1 & KKT & 2 & 2 & 2 & 2 & 1 \\
hs042 & KKT & 6 & 6 & 6 & 6 & 5 & KKT & 6 & 6 & 6 & 6 & 5 \\
hs043 & KKT & 9 & 9 & 8 & 8 & 7 & KKT & 9 & 9 & 8 & 8 & 7 \\
hs044 & KKT & 2 & 2 & 2 & 2 & 1 & KKT & 2 & 2 & 2 & 2 & 1 \\
hs046 & KKT & 19 & 19 & 19 & 19 & 18 & KKT & 19 & 19 & 19 & 19 & 18 \\
hs047 & KKT & 21 & 21 & 18 & 18 & 17 & KKT & 21 & 21 & 18 & 18 & 17 \\
hs048 & KKT & 2 & 2 & 2 & 2 & 1 & KKT & 2 & 2 & 2 & 2 & 1 \\
hs049 & KKT & 17 & 17 & 17 & 17 & 16 & KKT & 17 & 17 & 17 & 17 & 16 \\
hs050 & KKT & 9 & 9 & 9 & 9 & 8 & KKT & 9 & 9 & 9 & 9 & 8 \\
hs051 & KKT & 2 & 2 & 2 & 2 & 1 & KKT & 2 & 2 & 2 & 2 & 1 \\
hs052 & KKT & 2 & 2 & 2 & 2 & 1 & KKT & 2 & 2 & 2 & 2 & 1 \\
hs053 & KKT & 2 & 2 & 2 & 2 & 1 & KKT & 2 & 2 & 2 & 2 & 1 \\
hs054 & KKT & 2 & 2 & 2 & 2 & 1 & KKT & 2 & 2 & 2 & 2 & 1 \\
hs055 & KKT & 2 & 2 & 2 & 2 & 1 & KKT & 2 & 2 & 2 & 2 & 1 \\
hs056 & KKT & 15 & 16 & 16 & 16 & 16 & KKT & 15 & 16 & 16 & 16 & 16 \\
hs057 & KKT & 5 & 5 & 5 & 5 & 4 & KKT & 5 & 5 & 5 & 5 & 4 \\
hs059 & KKT & 11 & 12 & 10 & 10 & 10 & KKT & 11 & 12 & 10 & 10 & 10 \\
hs060 & KKT & 7 & 7 & 7 & 7 & 6 & KKT & 7 & 7 & 7 & 7 & 6 \\
hs061 & infeasible & 1 & 2 & 2 & 2 & 2 & infeasible & 1 & 2 & 2 & 2 & 2 \\
hs062 & KKT & 8 & 8 & 7 & 7 & 6 & KKT & 8 & 8 & 7 & 7 & 6 \\
hs063 & KKT & 9 & 11 & 10 & 10 & 10 & KKT & 8 & 10 & 8 & 8 & 9 \\
hs064 & KKT & 12 & 17 & 17 & 17 & 17 & KKT & 12 & 17 & 17 & 17 & 17 \\
hs065 & KKT & 5 & 5 & 5 & 5 & 4 & KKT & 5 & 5 & 5 & 5 & 4 \\
hs066 & KKT & 9 & 9 & 9 & 9 & 8 & KKT & 9 & 9 & 9 & 9 & 8 \\
hs067 & KKT & 12 & 12 & 12 & 12 & 11 & KKT & 12 & 12 & 12 & 12 & 11 \\
hs070 & KKT & 40 & 40 & 28 & 28 & 27 & KKT & 40 & 40 & 28 & 28 & 27 \\
hs071 & KKT & 6 & 6 & 6 & 6 & 5 & KKT & 6 & 6 & 6 & 6 & 5 \\
hs072 & KKT & 15 & 18 & 16 & 16 & 16 & KKT & 15 & 18 & 16 & 16 & 16 \\
hs073 & KKT & 4 & 4 & 4 & 4 & 3 & KKT & 4 & 4 & 4 & 4 & 3 \\
hs074 & KKT & 6 & 26 & 19 & 19 & 19 & KKT & 6 & 26 & 19 & 19 & 19 \\
hs075 & KKT & 5 & 25 & 18 & 18 & 18 & KKT & 5 & 25 & 18 & 18 & 18 \\
hs076 & KKT & 2 & 2 & 2 & 2 & 1 & KKT & 2 & 2 & 2 & 2 & 1 \\
hs077 & KKT & 11 & 11 & 10 & 10 & 9 & KKT & 11 & 11 & 10 & 10 & 9 \\
hs078 & KKT & 5 & 5 & 5 & 5 & 4 & KKT & 5 & 5 & 5 & 5 & 4 \\
hs079 & KKT & 5 & 5 & 5 & 5 & 4 & KKT & 5 & 5 & 5 & 5 & 4 \\
hs080 & KKT & 8 & 8 & 8 & 8 & 7 & KKT & 8 & 8 & 8 & 8 & 7 \\
hs081 & KKT & 111 & 113 & 58 & 58 & 59 & KKT & 29 & 29 & 18 & 18 & 17 \\
hs083 & KKT & 5 & 5 & 5 & 5 & 4 & KKT & 5 & 5 & 5 & 5 & 4 \\
hs084 & KKT & 11 & 11 & 8 & 8 & 7 & KKT & 11 & 11 & 8 & 8 & 7 \\
hs085 & error & -- & -- & -- & -- & -- & error & -- & -- & -- & -- & -- \\
hs086 & KKT & 5 & 5 & 5 & 5 & 4 & KKT & 5 & 5 & 5 & 5 & 4 \\
hs087 & KKT & 7 & 10 & 10 & 10 & 10 & KKT & 7 & 10 & 10 & 10 & 10 \\
hs088 & KKT & 22 & 29 & 20 & 20 & 21 & KKT & 21 & 24 & 21 & 21 & 22 \\
hs089 & KKT & 23 & 24 & 17 & 17 & 17 & KKT & 23 & 24 & 17 & 17 & 17 \\
hs090 & KKT & 81 & 89 & 50 & 50 & 51 & KKT & 74 & 82 & 50 & 50 & 51 \\
hs091 & KKT & 125 & 126 & 59 & 59 & 59 & KKT & 51 & 52 & 29 & 29 & 29 \\
hs092 & KKT & 96 & 106 & 50 & 50 & 51 & KKT & 40 & 51 & 28 & 28 & 30 \\
hs093 & infeasible & 3 & 3 & 3 & 3 & 3 & infeasible & 3 & 3 & 3 & 3 & 3 \\
hs095 & KKT & 3 & 3 & 3 & 3 & 2 & KKT & 3 & 3 & 3 & 3 & 2 \\
hs096 & KKT & 3 & 3 & 3 & 3 & 2 & KKT & 3 & 3 & 3 & 3 & 2 \\
hs097 & KKT & 7 & 7 & 7 & 7 & 6 & KKT & 7 & 7 & 7 & 7 & 6 \\
hs098 & KKT & 7 & 7 & 7 & 7 & 6 & KKT & 7 & 7 & 7 & 7 & 6 \\
hs099 & small feasible & 14 & 63 & 41 & 41 & 41 & small feasible & 14 & 63 & 41 & 41 & 41 \\
hs100 & KKT & 12 & 12 & 9 & 9 & 8 & KKT & 12 & 12 & 9 & 9 & 8 \\
hs100lnp & KKT & 14 & 14 & 12 & 12 & 11 & KKT & 12 & 12 & 9 & 9 & 8 \\
hs100mod & KKT & 12 & 12 & 9 & 9 & 8 & KKT & 12 & 12 & 9 & 9 & 8 \\
hs101 & KKT & 22 & 27 & 18 & 18 & 18 & KKT & 22 & 27 & 18 & 18 & 18 \\
hs102 & KKT & 17 & 18 & 15 & 15 & 15 & KKT & 17 & 18 & 15 & 15 & 15 \\
hs103 & KKT & 16 & 17 & 14 & 14 & 14 & KKT & 37 & 1678 & 843 & 843 & 848 \\
hs104 & KKT & 31 & 36 & 25 & 25 & 25 & KKT & 17 & 23 & 15 & 15 & 15 \\
hs106 & KKT & 771 & 771 & 386 & 386 & 385 & KKT & 16 & 16 & 15 & 15 & 14 \\
hs107 & KKT & 6 & 6 & 6 & 6 & 5 & KKT & 6 & 6 & 6 & 6 & 5 \\
hs108 & KKT & 21 & 21 & 14 & 14 & 13 & KKT & 25 & 25 & 19 & 19 & 18 \\
hs109 & KKT & 6 & 24 & 18 & 18 & 18 & KKT & 6 & 24 & 18 & 18 & 18 \\
hs111 & KKT & 35 & 35 & 21 & 21 & 20 & KKT & 40 & 40 & 25 & 25 & 24 \\
hs111lnp & KKT & 35 & 35 & 21 & 21 & 20 & KKT & 40 & 40 & 25 & 25 & 24 \\
hs112 & KKT & 12 & 12 & 12 & 12 & 11 & KKT & 12 & 12 & 12 & 12 & 11 \\
hs113 & KKT & 6 & 6 & 6 & 6 & 5 & KKT & 6 & 6 & 6 & 6 & 5 \\
hs114 & infeasible & 1 & 197 & 102 & 102 & 102 & infeasible & 1 & 197 & 102 & 102 & 102 \\
hs116 & KKT & 12 & 12 & 12 & 12 & 11 & KKT & 12 & 12 & 12 & 12 & 11 \\
hs117 & KKT & 6 & 6 & 6 & 6 & 5 & KKT & 6 & 6 & 6 & 6 & 5 \\
hs118 & KKT & 3 & 3 & 3 & 3 & 2 & KKT & 3 & 3 & 3 & 3 & 2 \\
hs119 & KKT & 7 & 7 & 7 & 7 & 6 & KKT & 7 & 7 & 7 & 7 & 6 \\
hs21mod & KKT & 2 & 2 & 2 & 2 & 1 & KKT & 2 & 2 & 2 & 2 & 1 \\
hs268 & KKT & 2 & 2 & 2 & 2 & 1 & KKT & 2 & 2 & 2 & 2 & 1 \\
hs35mod & KKT & 2 & 2 & 2 & 2 & 1 & KKT & 2 & 2 & 2 & 2 & 1 \\
hs44new & KKT & 2 & 2 & 2 & 2 & 1 & KKT & 2 & 2 & 2 & 2 & 1 \\
hs99exp & KKT & 12 & 42 & 32 & 32 & 32 & KKT & 12 & 42 & 32 & 32 & 32 \\
hubfit & KKT & 2 & 2 & 2 & 2 & 1 & KKT & 2 & 2 & 2 & 2 & 1 \\
hypcir & KKT & 6 & 6 & 6 & 6 & 5 & KKT & 6 & 7 & 6 & 6 & 6 \\
kiwcresc & KKT & 11 & 11 & 9 & 9 & 8 & KKT & 11 & 11 & 9 & 9 & 8 \\
lakes & KKT & 12 & 5456 & 2737 & 2737 & 2737 & KKT & 12 & 5456 & 2737 & 2737 & 2737 \\
launch & error & -- & -- & -- & -- & -- & infeasible & 1 & 1523 & 735 & 735 & 735 \\
lewispol & infeasible & 1 & 4 & 4 & 4 & 4 & infeasible & 1 & 4 & 4 & 4 & 4 \\
linspanh & KKT & 2 & 2 & 2 & 2 & 1 & KKT & 2 & 2 & 2 & 2 & 1 \\
loadbal & KKT & 8 & 8 & 8 & 8 & 7 & KKT & 8 & 8 & 8 & 8 & 7 \\
lootsma & KKT & 5 & 5 & 5 & 5 & 4 & KKT & 5 & 5 & 5 & 5 & 4 \\
lotschd & KKT & 3 & 3 & 3 & 3 & 2 & KKT & 3 & 3 & 3 & 3 & 2 \\
lsnnodoc & KKT & 7 & 7 & 7 & 7 & 6 & KKT & 7 & 7 & 7 & 7 & 6 \\
lsqfit & KKT & 2 & 2 & 2 & 2 & 1 & KKT & 2 & 2 & 2 & 2 & 1 \\
madsen & KKT & 14 & 14 & 11 & 11 & 10 & KKT & 14 & 14 & 11 & 11 & 10 \\
makela1 & KKT & 12 & 12 & 10 & 10 & 9 & KKT & 12 & 12 & 10 & 10 & 9 \\
makela2 & KKT & 5 & 5 & 5 & 5 & 4 & KKT & 5 & 5 & 5 & 5 & 4 \\
makela3 & KKT & 22 & 22 & 20 & 20 & 19 & KKT & 22 & 22 & 20 & 20 & 19 \\
makela4 & KKT & 3 & 3 & 3 & 3 & 2 & KKT & 3 & 3 & 3 & 3 & 2 \\
maratos & KKT & 10 & 10 & 9 & 9 & 8 & KKT & 10 & 10 & 9 & 9 & 8 \\
matrix2 & KKT & 12 & 12 & 12 & 12 & 11 & KKT & 12 & 12 & 12 & 12 & 11 \\
mconcon & KKT & 5 & 11 & 10 & 10 & 10 & KKT & 7 & 16 & 12 & 12 & 16 \\
mifflin1 & KKT & 10 & 10 & 9 & 9 & 8 & KKT & 10 & 10 & 9 & 9 & 8 \\
mifflin2 & KKT & 11 & 11 & 8 & 8 & 7 & KKT & 10 & 10 & 8 & 8 & 7 \\
minmaxbd & KKT & 32 & 45 & 25 & 25 & 27 & KKT & 22 & 25 & 18 & 18 & 18 \\
minmaxrb & KKT & 3 & 3 & 3 & 3 & 2 & KKT & 3 & 3 & 3 & 3 & 2 \\
mistake & KKT & 27 & 27 & 17 & 17 & 16 & KKT & 22 & 22 & 19 & 19 & 18 \\
model & KKT & 2 & 2 & 2 & 2 & 1 & KKT & 2 & 2 & 2 & 2 & 1 \\
mwright & KKT & 9 & 9 & 7 & 7 & 6 & KKT & 9 & 9 & 7 & 7 & 6 \\
nuffield\footnote{full name: nuffield\_continuum, which this table is too narrow to contain.} & KKT & 5 & 5 & 5 & 5 & 4 & KKT & 5 & 5 & 5 & 5 & 4 \\
odfits & KKT & 7 & 7 & 7 & 7 & 6 & KKT & 7 & 7 & 7 & 7 & 6 \\
optcntrl & KKT & 4 & 4 & 4 & 4 & 3 & KKT & 4 & 4 & 4 & 4 & 3 \\
optmass & KKT & 9 & 9 & 7 & 7 & 6 & KKT & 9 & 9 & 7 & 7 & 6 \\
optprloc & KKT & 18 & 18 & 10 & 10 & 9 & KKT & 16 & 16 & 9 & 9 & 8 \\
orthregb & KKT & 2 & 2 & 2 & 2 & 1 & KKT & 2 & 2 & 2 & 2 & 1 \\
orthrege & KKT & 2334 & 2364 & 1183 & 1183 & 1185 & KKT & 2896 & 2898 & 1450 & 1450 & 1450 \\
pentagon & KKT & 8 & 8 & 6 & 6 & 5 & KKT & 8 & 8 & 6 & 6 & 5 \\
polak1 & KKT & 8 & 8 & 8 & 8 & 7 & KKT & 8 & 8 & 8 & 8 & 7 \\
polak2 & KKT & 22 & 31 & 20 & 20 & 23 & KKT & 46 & 55 & 32 & 32 & 35 \\
polak3 & KKT & 21 & 26 & 15 & 15 & 15 & KKT & 21 & 26 & 15 & 15 & 15 \\
polak4 & KKT & 5 & 5 & 5 & 5 & 4 & KKT & 5 & 5 & 5 & 5 & 4 \\
polak5 & KKT & 43 & 57 & 45 & 45 & 45 & KKT & 82 & 96 & 47 & 47 & 47 \\
polak6 & KKT & 31 & 34 & 17 & 17 & 18 & KKT & 43 & 48 & 23 & 23 & 25 \\
portfl1 & KKT & 2 & 2 & 2 & 2 & 1 & KKT & 2 & 2 & 2 & 2 & 1 \\
portfl2 & KKT & 2 & 2 & 2 & 2 & 1 & KKT & 2 & 2 & 2 & 2 & 1 \\
portfl3 & KKT & 2 & 2 & 2 & 2 & 1 & KKT & 2 & 2 & 2 & 2 & 1 \\
portfl4 & KKT & 2 & 2 & 2 & 2 & 1 & KKT & 2 & 2 & 2 & 2 & 1 \\
portfl6 & KKT & 2 & 2 & 2 & 2 & 1 & KKT & 2 & 2 & 2 & 2 & 1 \\
powellbs & KKT & 12 & 12 & 12 & 12 & 11 & KKT & 10 & 115 & 60 & 60 & 62 \\
powellsq & KKT & 21 & 35 & 26 & 26 & 26 & KKT & 80 & 277 & 135 & 135 & 216 \\
prodpl0 & KKT & 9 & 9 & 9 & 9 & 8 & KKT & 9 & 9 & 9 & 9 & 8 \\
prodpl1 & KKT & 7 & 7 & 7 & 7 & 6 & KKT & 7 & 7 & 7 & 7 & 6 \\
recipe & KKT & 2 & 2 & 2 & 2 & 1 & KKT & 2 & 2 & 2 & 2 & 1 \\
res & KKT & 2 & 2 & 2 & 2 & 1 & small feasible & 3 & 3 & 2 & 2 & 1 \\
rk23 & KKT & 5 & 5 & 5 & 5 & 4 & KKT & 7 & 9 & 7 & 7 & 7 \\
robot & KKT & 14 & 21 & 12 & 12 & 12 & KKT & 14 & 21 & 12 & 12 & 12 \\
rosenmmx & KKT & 24 & 27 & 14 & 14 & 14 & KKT & 36 & 40 & 19 & 19 & 20 \\
s365mod & KKT & 32 & 48 & 27 & 27 & 32 & KKT & 23 & 29 & 18 & 18 & 18 \\
simpllpa & KKT & 2 & 2 & 2 & 2 & 1 & KKT & 2 & 2 & 2 & 2 & 1 \\
simpllpb & KKT & 2 & 2 & 2 & 2 & 1 & KKT & 2 & 2 & 2 & 2 & 1 \\
snake & KKT & 3 & 3 & 3 & 3 & 2 & KKT & 3 & 3 & 3 & 3 & 2 \\
spanhyd & KKT & 4 & 4 & 4 & 4 & 3 & KKT & 4 & 4 & 4 & 4 & 3 \\
spiral & KKT & 26 & 26 & 21 & 21 & 20 & KKT & 113 & 117 & 88 & 88 & 90 \\
ssnlbeam & KKT & 5 & 5 & 5 & 5 & 4 & KKT & 5 & 5 & 5 & 5 & 4 \\
stancmin & KKT & 2 & 2 & 2 & 2 & 1 & KKT & 2 & 2 & 2 & 2 & 1 \\
supersim & KKT & 2 & 2 & 2 & 2 & 1 & KKT & 2 & 2 & 2 & 2 & 1 \\
swopf & KKT & 5 & 5 & 5 & 5 & 4 & KKT & 5 & 5 & 5 & 5 & 4 \\
synthes1 & KKT & 5 & 5 & 5 & 5 & 4 & KKT & 5 & 5 & 5 & 5 & 4 \\
tame & KKT & 2 & 2 & 2 & 2 & 1 & KKT & 2 & 2 & 2 & 2 & 1 \\
try-b & KKT & 8 & 8 & 8 & 8 & 7 & KKT & 8 & 8 & 8 & 8 & 7 \\
twobars & KKT & 8 & 8 & 8 & 8 & 7 & KKT & 8 & 8 & 8 & 8 & 7 \\
vanderm4 & infeasible & 1 & 17 & 17 & 17 & 17 & infeasible & 1 & 17 & 17 & 17 & 17 \\
womflet & KKT & 33 & 41 & 27 & 27 & 27 & KKT & 42 & 46 & 24 & 24 & 25 \\
zangwil3 & KKT & 2 & 2 & 2 & 2 & 1 & KKT & 2 & 2 & 2 & 2 & 1 \\
zecevic2 & KKT & 2 & 2 & 2 & 2 & 1 & KKT & 2 & 2 & 2 & 2 & 1 \\
zecevic3 & KKT & 10 & 15 & 10 & 10 & 10 & KKT & 10 & 15 & 10 & 10 & 10 \\
zecevic4 & KKT & 6 & 6 & 6 & 6 & 5 & KKT & 6 & 6 & 6 & 6 & 5 \\
zigzag & KKT & 11 & 11 & 11 & 11 & 10 & KKT & 11 & 11 & 11 & 11 & 10 \\
zy2 & KKT & 5 & 5 & 5 & 5 & 4 & KKT & 5 & 5 & 5 & 5 & 4 \\
\hline
\end{longtable}

\section{Detailed Results on \texttt{CUTEst} problems: Number of Function and Derivative Evaluations for Line-Search Algorithms}\label{App:LS}
\footnotesize
\renewcommand{\arraystretch}{1}
\setlength{\tabcolsep}{4pt}
\begin{longtable}{|c||c|ccccc||c|ccccc|}
\hline
 & \multicolumn{6}{c|}{line-search funnel method}           & \multicolumn{6}{c|}{line-search filter method} \\
\hline
instance & status & $f$ & $c$ & $\nabla f$ & $\nabla c$ & $\nabla^2 \mathcal{L}$ & status & $f$ & $c$ & $\nabla f$ & $\nabla c$ & $\nabla^2 \mathcal{L}$ \\
\hline    
aircrfta & KKT  & 4 & 4 & 4 & 4 & 3 & KKT  & 4 & 4 & 4 & 4 & 3 \\
airport & KKT  & 13 & 13 & 13 & 13 & 12 & KKT  & 13 & 13 & 13 & 13 & 12 \\
aljazzaf & KKT  & 58 & 58 & 34 & 34 & 33 & KKT  & 131 & 171 & 40 & 40 & 40 \\
allinitc & KKT  & 29 & 29 & 29 & 29 & 28 & KKT  & 29 & 29 & 29 & 29 & 28 \\
alsotame & KKT  & 5 & 5 & 5 & 5 & 4 & KKT  & 5 & 5 & 5 & 5 & 4 \\
argauss & infeasible & 1 & 3 & 3 & 3 & 3 & infeasible & 1 & 3 & 3 & 3 & 3 \\
avgasa & KKT  & 3 & 3 & 3 & 3 & 2 & KKT  & 3 & 3 & 3 & 3 & 2 \\
avgasb & KKT  & 3 & 3 & 3 & 3 & 2 & KKT  & 3 & 3 & 3 & 3 & 2 \\
avion2 & KKT  & 3 & 3 & 3 & 3 & 2 & KKT  & 3 & 3 & 3 & 3 & 2 \\
batch & KKT  & 9 & 9 & 9 & 9 & 8 & KKT  & 9 & 9 & 9 & 9 & 8 \\
biggsc4 & KKT  & 4 & 4 & 4 & 4 & 3 & KKT  & 4 & 4 & 4 & 4 & 3 \\
booth & KKT  & 2 & 2 & 2 & 2 & 1 & KKT  & 2 & 2 & 2 & 2 & 1 \\
bt1 & infeasible & 2 & 2 & 2 & 2 & 2 & infeasible & 2 & 2 & 2 & 2 & 2 \\
bt10 & KKT  & 7 & 7 & 7 & 7 & 6 & KKT  & 7 & 7 & 7 & 7 & 6 \\
bt11 & KKT  & 7 & 7 & 7 & 7 & 6 & KKT  & 7 & 7 & 7 & 7 & 6 \\
bt12 & KKT  & 5 & 5 & 5 & 5 & 4 & KKT  & 5 & 5 & 5 & 5 & 4 \\
bt13 & KKT  & 24 & 24 & 21 & 21 & 20 & KKT  & 24 & 24 & 21 & 21 & 20 \\
bt2 & KKT  & 13 & 13 & 13 & 13 & 12 & KKT  & 13 & 13 & 13 & 13 & 12 \\
bt3 & KKT  & 3 & 3 & 3 & 3 & 2 & KKT  & 3 & 3 & 3 & 3 & 2 \\
bt4 & KKT  & 50 & 50 & 50 & 50 & 49 & KKT  & 50 & 50 & 50 & 50 & 49 \\
bt5 & KKT  & 37 & 37 & 37 & 37 & 36 & KKT  & 37 & 37 & 37 & 37 & 36 \\
bt6 & KKT  & 30 & 30 & 30 & 30 & 29 & KKT  & 30 & 30 & 30 & 30 & 29 \\
bt7 & KKT  & 87 & 87 & 53 & 53 & 52 & KKT  & 109 & 109 & 52 & 52 & 52 \\
bt8 & KKT  & 28 & 28 & 28 & 28 & 27 & KKT  & 28 & 28 & 28 & 28 & 27 \\
bt9 & KKT  & 63 & 63 & 22 & 22 & 21 & KKT  & 52 & 52 & 20 & 20 & 19 \\
byrdsphr & error & -- & -- & -- & -- & -- & KKT  & 5495 & 5589 & 453 & 453 & 498 \\
cantilvr & KKT  & 21 & 21 & 14 & 14 & 13 & KKT  & 24 & 24 & 14 & 14 & 13 \\
catena & KKT  & 340 & 340 & 31 & 31 & 32 & KKT  & 244 & 244 & 25 & 25 & 25 \\
cb2 & KKT  & 7 & 7 & 7 & 7 & 6 & KKT  & 7 & 7 & 7 & 7 & 6 \\
cb3 & KKT  & 7 & 7 & 7 & 7 & 6 & KKT  & 7 & 7 & 7 & 7 & 6 \\
chaconn1 & KKT  & 5 & 5 & 5 & 5 & 4 & KKT  & 5 & 5 & 5 & 5 & 4 \\
chaconn2 & KKT  & 5 & 5 & 5 & 5 & 4 & KKT  & 5 & 5 & 5 & 5 & 4 \\
cluster & KKT  & 9 & 9 & 9 & 9 & 8 & KKT  & 9 & 9 & 9 & 9 & 8 \\
concon & KKT  & 5 & 5 & 5 & 5 & 4 & KKT  & 5 & 5 & 5 & 5 & 4 \\
congigmz & KKT  & 5 & 5 & 5 & 5 & 4 & KKT  & 5 & 5 & 5 & 5 & 4 \\
coolhans & KKT  & 6 & 6 & 6 & 6 & 5 & KKT  & 57 & 65 & 11 & 11 & 11 \\
core1 & KKT  & 254 & 254 & 254 & 254 & 253 & KKT  & 254 & 254 & 254 & 254 & 253 \\
coshfun & error & -- & -- & -- & -- & -- & error & -- & -- & -- & -- & -- \\
cresc4 & error & -- & -- & -- & -- & -- & error & -- & -- & -- & -- & -- \\
csfi1 & KKT  & 14 & 14 & 11 & 11 & 10 & KKT & 47 & 47 & 27 & 27 & 26 \\
csfi2 & KKT  & 72 & 72 & 31 & 31 & 30 & KKT & 88 & 88 & 31 & 31 & 30 \\
dallass & KKT  & 17 & 17 & 14 & 14 & 13 & KKT  & 17 & 17 & 14 & 14 & 13 \\
deconvc & KKT  & 82 & 82 & 82 & 82 & 81 & KKT  & 82 & 82 & 82 & 82 & 81 \\
degenlpa & KKT  & 2 & 2 & 2 & 2 & 1 & KKT  & 2 & 2 & 2 & 2 & 1 \\
degenlpb & KKT  & 2 & 2 & 2 & 2 & 1 & KKT  & 2 & 2 & 2 & 2 & 1 \\
demymalo & KKT  & 17 & 17 & 8 & 8 & 7 & KKT  & 17 & 17 & 8 & 8 & 7 \\
dipigri & KKT  & 17 & 17 & 8 & 8 & 7 & KKT  & 17 & 17 & 8 & 8 & 7 \\
disc2 & infeasible & 2 & 198 & 198 & 198 & 198 & infeasible & 2 & 198 & 198 & 198 & 198 \\
discs & KKT  & 2758 & 2767 & 2767 & 2767 & 2767 & KKT  & 2758 & 2767 & 2767 & 2767 & 2767 \\
dixchlng & KKT  & 224 & 224 & 208 & 208 & 207 & error & -- & -- & -- & -- & -- \\
dnieper & KKT  & 577 & 577 & 577 & 577 & 576 & KKT  & 577 & 577 & 577 & 577 & 576 \\
dual1 & KKT  & 3 & 3 & 3 & 3 & 2 & KKT  & 3 & 3 & 3 & 3 & 2 \\
dual2 & KKT  & 3 & 3 & 3 & 3 & 2 & KKT  & 3 & 3 & 3 & 3 & 2 \\
dual4 & KKT  & 3 & 3 & 3 & 3 & 2 & KKT  & 3 & 3 & 3 & 3 & 2 \\
dualc1 & KKT  & 2 & 2 & 2 & 2 & 1 & KKT  & 2 & 2 & 2 & 2 & 1 \\
dualc2 & KKT  & 2 & 2 & 2 & 2 & 1 & KKT  & 2 & 2 & 2 & 2 & 1 \\
dualc5 & error & -- & -- & -- & -- & -- & error & -- & -- & -- & -- & -- \\
dualc8 & KKT  & 2 & 2 & 2 & 2 & 1 & KKT  & 2 & 2 & 2 & 2 & 1 \\
eigencco & KKT  & 13 & 13 & 13 & 13 & 12 & KKT  & 13 & 13 & 13 & 13 & 12 \\
eigmaxc & KKT  & 6 & 7 & 7 & 7 & 7 & KKT  & 6 & 7 & 7 & 7 & 7 \\
eigminc & KKT  & 6 & 7 & 7 & 7 & 7 & KKT  & 6 & 7 & 7 & 7 & 7 \\
expfita & KKT  & 14 & 14 & 14 & 14 & 13 & KKT  & 14 & 14 & 14 & 14 & 13 \\
extrasim & KKT  & 3 & 3 & 3 & 3 & 2 & KKT  & 3 & 3 & 3 & 3 & 2 \\
fccu & KKT  & 3 & 3 & 3 & 3 & 2 & KKT  & 3 & 3 & 3 & 3 & 2 \\
fletcher & infeasible & 2 & 2 & 2 & 2 & 2 & infeasible & 2 & 2 & 2 & 2 & 2 \\
genhs28 & KKT  & 3 & 3 & 3 & 3 & 2 & KKT  & 3 & 3 & 3 & 3 & 2 \\
gigomez1 & KKT  & 17 & 17 & 8 & 8 & 7 & KKT  & 17 & 17 & 8 & 8 & 7 \\
goffin & KKT  & 3 & 3 & 3 & 3 & 2 & KKT  & 3 & 3 & 3 & 3 & 2 \\
gottfr & KKT  & 8 & 8 & 8 & 8 & 7 & KKT  & 8 & 8 & 6 & 6 & 5 \\
gridnetg & KKT  & 4 & 4 & 4 & 4 & 3 & KKT  & 4 & 4 & 4 & 4 & 3 \\
gridneth & KKT  & 5 & 5 & 5 & 5 & 4 & KKT  & 5 & 5 & 5 & 5 & 4 \\
gridneti & KKT  & 5 & 5 & 5 & 5 & 4 & KKT  & 5 & 5 & 5 & 5 & 4 \\
haifas & KKT  & 31 & 31 & 9 & 9 & 8 & KKT  & 31 & 31 & 9 & 9 & 8 \\
haldmads & KKT  & 143 & 143 & 136 & 136 & 135 & error & -- & -- & -- & -- & -- \\
hatfldf & KKT  & 9 & 9 & 8 & 8 & 7 & infeasible & 47 & 3267 & 1524 & 1524 & 1524 \\
hatfldg & KKT  & 17 & 17 & 9 & 9 & 8 & KKT  & 21 & 21 & 9 & 9 & 8 \\
hatfldh & KKT  & 4 & 4 & 4 & 4 & 3 & KKT  & 4 & 4 & 4 & 4 & 3 \\
heart6 & KKT  & 3 & 2846 & 2843 & 2843 & 2843 & KKT  & 3 & 2846 & 2843 & 2843 & 2843 \\
heart8 & KKT  & 5 & 26 & 19 & 19 & 19 & KKT  & 5 & 26 & 19 & 19 & 19 \\
himmelba & KKT  & 2 & 2 & 2 & 2 & 1 & KKT  & 2 & 2 & 2 & 2 & 1 \\
himmelbc & KKT  & 6 & 6 & 6 & 6 & 5 & KKT  & 7 & 7 & 6 & 6 & 5 \\
himmelbd & infeasible & 134 & 181 & 37 & 37 & 37 & infeasible & 57 & 109 & 21 & 21 & 21 \\
himmelbe & KKT  & 3 & 3 & 3 & 3 & 2 & KKT  & 3 & 3 & 3 & 3 & 2 \\
himmelbk & KKT  & 708 & 708 & 708 & 708 & 707 & KKT  & 708 & 708 & 708 & 708 & 707 \\
himmelp2 & KKT  & 10 & 10 & 10 & 10 & 9 & KKT  & 10 & 10 & 10 & 10 & 9 \\
himmelp3 & KKT  & 5 & 5 & 5 & 5 & 4 & KKT  & 5 & 5 & 5 & 5 & 4 \\
himmelp4 & KKT  & 6 & 6 & 6 & 6 & 5 & KKT  & 6 & 6 & 6 & 6 & 5 \\
himmelp5 & KKT  & 16 & 16 & 16 & 16 & 15 & KKT  & 16 & 16 & 16 & 16 & 15 \\
himmelp6 & KKT  & 2 & 2 & 2 & 2 & 1 & KKT  & 2 & 2 & 2 & 2 & 1 \\
hong & KKT  & 5 & 5 & 5 & 5 & 4 & KKT  & 5 & 5 & 5 & 5 & 4 \\
hs006 & KKT  & 5 & 5 & 5 & 5 & 4 & KKT  & 5 & 5 & 5 & 5 & 4 \\
hs007 & KKT  & 16 & 16 & 11 & 11 & 10 & KKT  & 16 & 16 & 11 & 11 & 10 \\
hs008 & KKT  & 6 & 6 & 6 & 6 & 5 & KKT  & 6 & 6 & 6 & 6 & 5 \\
hs009 & KKT  & 5 & 5 & 4 & 4 & 3 & KKT  & 5 & 5 & 4 & 4 & 3 \\
hs010 & KKT  & 39 & 39 & 15 & 15 & 14 & KKT  & 52 & 52 & 18 & 18 & 17 \\
hs011 & KKT  & 6 & 6 & 6 & 6 & 5 & KKT  & 6 & 6 & 6 & 6 & 5 \\
hs012 & KKT  & 9 & 9 & 7 & 7 & 6 & KKT  & 9 & 9 & 7 & 7 & 6 \\
hs013 & KKT  & 26 & 26 & 26 & 26 & 25 & KKT  & 26 & 26 & 26 & 26 & 25 \\
hs014 & KKT  & 6 & 6 & 6 & 6 & 5 & KKT  & 6 & 6 & 6 & 6 & 5 \\
hs015 & KKT  & 4 & 9 & 9 & 9 & 9 & KKT  & 4 & 9 & 9 & 9 & 9 \\
hs016 & KKT  & 6 & 6 & 6 & 6 & 5 & KKT  & 6 & 6 & 6 & 6 & 5 \\
hs017 & KKT  & 10 & 10 & 9 & 9 & 8 & KKT  & 10 & 10 & 9 & 9 & 8 \\
hs018 & KKT  & 7 & 7 & 7 & 7 & 6 & KKT  & 7 & 7 & 7 & 7 & 6 \\
hs019 & KKT  & 7 & 7 & 7 & 7 & 6 & KKT  & 7 & 7 & 7 & 7 & 6 \\
hs020 & KKT  & 5 & 5 & 5 & 5 & 4 & KKT  & 5 & 5 & 5 & 5 & 4 \\
hs021 & KKT  & 3 & 3 & 3 & 3 & 2 & KKT  & 3 & 3 & 3 & 3 & 2 \\
hs022 & KKT  & 3 & 3 & 3 & 3 & 2 & KKT  & 3 & 3 & 3 & 3 & 2 \\
hs023 & KKT  & 7 & 7 & 7 & 7 & 6 & KKT  & 7 & 7 & 7 & 7 & 6 \\
hs024 & KKT  & 16 & 16 & 16 & 16 & 15 & KKT  & 16 & 16 & 16 & 16 & 15 \\
hs026 & KKT  & 19 & 19 & 19 & 19 & 18 & KKT  & 19 & 19 & 19 & 19 & 18 \\
hs027 & KKT  & 14 & 14 & 14 & 14 & 13 & KKT  & 38 & 38 & 16 & 16 & 15 \\
hs028 & KKT  & 3 & 3 & 3 & 3 & 2 & KKT  & 3 & 3 & 3 & 3 & 2 \\
hs029 & KKT  & 25 & 25 & 25 & 25 & 24 & KKT  & 25 & 25 & 25 & 25 & 24 \\
hs030 & KKT  & 3 & 3 & 3 & 3 & 2 & KKT  & 3 & 3 & 3 & 3 & 2 \\
hs031 & KKT  & 7 & 7 & 7 & 7 & 6 & KKT  & 7 & 7 & 7 & 7 & 6 \\
hs032 & KKT  & 3 & 3 & 3 & 3 & 2 & KKT  & 3 & 3 & 3 & 3 & 2 \\
hs033 & KKT  & 16 & 16 & 16 & 16 & 15 & KKT  & 16 & 16 & 16 & 16 & 15 \\
hs034 & KKT  & 8 & 8 & 8 & 8 & 7 & KKT  & 8 & 8 & 8 & 8 & 7 \\
hs035 & KKT  & 3 & 3 & 3 & 3 & 2 & KKT  & 3 & 3 & 3 & 3 & 2 \\
hs036 & KKT  & 7 & 7 & 7 & 7 & 6 & KKT  & 7 & 7 & 7 & 7 & 6 \\
hs037 & KKT  & 138 & 138 & 132 & 132 & 131 & KKT  & 138 & 138 & 132 & 132 & 131 \\
hs039 & KKT  & 63 & 63 & 22 & 22 & 21 & KKT  & 52 & 52 & 20 & 20 & 19 \\
hs040 & KKT  & 19 & 19 & 19 & 19 & 18 & KKT  & 19 & 19 & 19 & 19 & 18 \\
hs041 & KKT  & 2 & 2 & 2 & 2 & 1 & KKT  & 2 & 2 & 2 & 2 & 1 \\
hs042 & KKT  & 6 & 6 & 6 & 6 & 5 & KKT  & 6 & 6 & 6 & 6 & 5 \\
hs043 & KKT  & 8 & 8 & 7 & 7 & 6 & KKT  & 8 & 8 & 7 & 7 & 6 \\
hs044 & KKT  & 10 & 10 & 10 & 10 & 9 & KKT  & 10 & 10 & 10 & 10 & 9 \\
hs046 & KKT  & 18 & 18 & 18 & 18 & 17 & KKT  & 18 & 18 & 18 & 18 & 17 \\
hs047 & KKT  & 18 & 18 & 18 & 18 & 17 & KKT  & 18 & 18 & 18 & 18 & 17 \\
hs048 & KKT  & 3 & 3 & 3 & 3 & 2 & KKT  & 3 & 3 & 3 & 3 & 2 \\
hs049 & KKT  & 18 & 18 & 18 & 18 & 17 & KKT  & 18 & 18 & 18 & 18 & 17 \\
hs050 & KKT  & 10 & 10 & 10 & 10 & 9 & KKT  & 10 & 10 & 10 & 10 & 9 \\
hs051 & KKT  & 3 & 3 & 3 & 3 & 2 & KKT  & 3 & 3 & 3 & 3 & 2 \\
hs052 & KKT  & 3 & 3 & 3 & 3 & 2 & KKT  & 3 & 3 & 3 & 3 & 2 \\
hs053 & KKT  & 3 & 3 & 3 & 3 & 2 & KKT  & 3 & 3 & 3 & 3 & 2 \\
hs054 & KKT  & 3 & 3 & 3 & 3 & 2 & KKT  & 3 & 3 & 3 & 3 & 2 \\
hs055 & KKT  & 2 & 2 & 2 & 2 & 1 & KKT  & 2 & 2 & 2 & 2 & 1 \\
hs056 & KKT  & 102 & 102 & 102 & 102 & 101 & KKT  & 102 & 102 & 102 & 102 & 101 \\
hs057 & KKT  & 26 & 26 & 26 & 26 & 25 & KKT  & 26 & 26 & 26 & 26 & 25 \\
hs059 & KKT  & 18 & 18 & 12 & 12 & 11 & KKT  & 18 & 18 & 12 & 12 & 11 \\
hs060 & KKT  & 7 & 7 & 7 & 7 & 6 & KKT  & 7 & 7 & 7 & 7 & 6 \\
hs061 & infeasible & 2 & 3 & 3 & 3 & 3 & infeasible & 2 & 3 & 3 & 3 & 3 \\
hs062 & KKT  & 7 & 7 & 6 & 6 & 5 & KKT  & 7 & 7 & 6 & 6 & 5 \\
hs063 & KKT  & 42 & 44 & 43 & 43 & 43 & KKT  & 42 & 44 & 43 & 43 & 43 \\
hs064 & KKT  & 16 & 16 & 16 & 16 & 15 & KKT  & 16 & 16 & 16 & 16 & 15 \\
hs065 & KKT  & 5 & 5 & 5 & 5 & 4 & KKT  & 5 & 5 & 5 & 5 & 4 \\
hs066 & KKT  & 9 & 9 & 9 & 9 & 8 & KKT  & 9 & 9 & 9 & 9 & 8 \\
hs067 & error & -- & -- & -- & -- & -- & error & -- & -- & -- & -- & -- \\
hs070 & KKT  & 25 & 25 & 19 & 19 & 18 & KKT  & 25 & 25 & 19 & 19 & 18 \\
hs071 & KKT  & 47 & 47 & 47 & 47 & 46 & KKT  & 47 & 47 & 47 & 47 & 46 \\
hs072 & KKT  & 17 & 17 & 15 & 15 & 14 & KKT  & 18 & 18 & 15 & 15 & 14 \\
hs073 & KKT  & 4 & 4 & 4 & 4 & 3 & KKT  & 4 & 4 & 4 & 4 & 3 \\
hs074 & KKT  & 6 & 6 & 6 & 6 & 5 & KKT  & 6 & 6 & 6 & 6 & 5 \\
hs075 & KKT  & 5 & 5 & 5 & 5 & 4 & KKT  & 5 & 5 & 5 & 5 & 4 \\
hs076 & KKT  & 3 & 3 & 3 & 3 & 2 & KKT  & 3 & 3 & 3 & 3 & 2 \\
hs077 & KKT  & 12 & 12 & 11 & 11 & 10 & KKT  & 12 & 12 & 11 & 11 & 10 \\
hs078 & KKT  & 35 & 35 & 35 & 35 & 34 & KKT  & 35 & 35 & 35 & 35 & 34 \\
hs079 & KKT  & 5 & 5 & 5 & 5 & 4 & KKT  & 5 & 5 & 5 & 5 & 4 \\
hs080 & KKT  & 8 & 8 & 8 & 8 & 7 & KKT  & 8 & 8 & 8 & 8 & 7 \\
hs081 & error & -- & -- & -- & -- & -- & error & -- & -- & -- & -- & -- \\
hs083 & KKT  & 5 & 5 & 5 & 5 & 4 & KKT  & 5 & 5 & 5 & 5 & 4 \\
hs084 & KKT  & 2859 & 2859 & 2859 & 2859 & 2858 & KKT  & 2859 & 2859 & 2859 & 2859 & 2858 \\
hs085 & infeasible & 1 & 34 & 23 & 23 & 23 & infeasible & 1 & 34 & 23 & 23 & 23 \\
hs086 & KKT  & 5 & 5 & 5 & 5 & 4 & KKT  & 5 & 5 & 5 & 5 & 4 \\
hs087 & KKT  & 20 & 20 & 20 & 20 & 19 & KKT  & 20 & 20 & 20 & 20 & 19 \\
hs088 & infeasible & 6 & 6 & 6 & 6 & 6 & KKT  & 19 & 19 & 17 & 17 & 16 \\
hs089 & KKT  & 24 & 24 & 19 & 19 & 18 & KKT  & 24 & 24 & 19 & 19 & 18 \\
hs090 & infeasible & 4 & 4 & 4 & 4 & 4 & KKT  & 18 & 18 & 12 & 12 & 11 \\
hs091 & KKT  & 32 & 32 & 29 & 29 & 28 & KKT  & 39 & 39 & 25 & 25 & 24 \\
hs092 & infeasible & 4 & 4 & 4 & 4 & 4 & KKT  & 31 & 31 & 15 & 15 & 14 \\
hs093 & KKT  & 1558 & 1558 & 1558 & 1558 & 1557 & KKT  & 1558 & 1558 & 1558 & 1558 & 1557 \\
hs095 & KKT  & 3 & 3 & 3 & 3 & 2 & KKT  & 3 & 3 & 3 & 3 & 2 \\
hs096 & KKT  & 3 & 3 & 3 & 3 & 2 & KKT  & 3 & 3 & 3 & 3 & 2 \\
hs097 & KKT  & 63 & 63 & 63 & 63 & 62 & KKT  & 63 & 63 & 63 & 63 & 62 \\
hs098 & KKT  & 63 & 63 & 63 & 63 & 62 & KKT  & 63 & 63 & 63 & 63 & 62 \\
hs099 & KKT  & 7 & 7 & 7 & 7 & 6 & KKT  & 7 & 7 & 7 & 7 & 6 \\
hs100 & KKT  & 17 & 17 & 8 & 8 & 7 & KKT  & 17 & 17 & 8 & 8 & 7 \\
hs100lnp & KKT  & 27 & 27 & 21 & 21 & 20 & KKT  & 24 & 24 & 17 & 17 & 16 \\
hs100mod & KKT  & 25 & 25 & 10 & 10 & 9 & KKT  & 25 & 25 & 10 & 10 & 9 \\
hs101 & KKT  & 24 & 24 & 18 & 18 & 17 & KKT  & 24 & 24 & 18 & 18 & 17 \\
hs102 & KKT  & 21 & 21 & 19 & 19 & 18 & KKT  & 87 & 96 & 26 & 26 & 26 \\
hs103 & KKT  & 14 & 14 & 12 & 12 & 11 & KKT  & 35 & 35 & 27 & 27 & 26 \\
hs104 & KKT  & 41 & 41 & 41 & 41 & 40 & KKT  & 126 & 130 & 51 & 51 & 51 \\
hs106 & KKT  & 748 & 748 & 744 & 744 & 743 & KKT  & 724 & 724 & 721 & 721 & 720 \\
hs107 & KKT  & 11 & 11 & 11 & 11 & 10 & KKT  & 11 & 11 & 11 & 11 & 10 \\
hs108 & KKT  & 11 & 11 & 11 & 11 & 10 & KKT  & 11 & 11 & 11 & 11 & 10 \\
hs109 & KKT  & 45 & 46 & 46 & 46 & 46 & KKT  & 45 & 46 & 46 & 46 & 46 \\
hs111 & KKT  & 17 & 17 & 17 & 17 & 16 & KKT  & 17 & 17 & 17 & 17 & 16 \\
hs111lnp & KKT  & 17 & 17 & 17 & 17 & 16 & KKT  & 17 & 17 & 17 & 17 & 16 \\
hs112 & KKT  & 12 & 12 & 12 & 12 & 11 & KKT  & 12 & 12 & 12 & 12 & 11 \\
hs113 & KKT  & 6 & 6 & 6 & 6 & 5 & KKT  & 6 & 6 & 6 & 6 & 5 \\
hs114 & error & -- & -- & -- & -- & -- & error & -- & -- & -- & -- & -- \\
hs116 & error & -- & -- & -- & -- & -- & error & -- & -- & -- & -- & -- \\
hs117 & KKT  & 8 & 8 & 7 & 7 & 6 & KKT  & 8 & 8 & 7 & 7 & 6 \\
hs118 & KKT  & 3 & 3 & 3 & 3 & 2 & KKT  & 3 & 3 & 3 & 3 & 2 \\
hs119 & KKT  & 7 & 7 & 7 & 7 & 6 & KKT  & 7 & 7 & 7 & 7 & 6 \\
hs21mod & KKT  & 3 & 3 & 3 & 3 & 2 & KKT  & 3 & 3 & 3 & 3 & 2 \\
hs268 & KKT  & 4 & 4 & 4 & 4 & 3 & KKT  & 4 & 4 & 4 & 4 & 3 \\
hs35mod & KKT  & 3 & 3 & 3 & 3 & 2 & KKT  & 3 & 3 & 3 & 3 & 2 \\
hs44new & KKT  & 7 & 7 & 7 & 7 & 6 & KKT  & 7 & 7 & 7 & 7 & 6 \\
hs99exp & error & -- & -- & -- & -- & -- & error & -- & -- & -- & -- & -- \\
hubfit & KKT  & 3 & 3 & 3 & 3 & 2 & KKT  & 3 & 3 & 3 & 3 & 2 \\
hypcir & KKT  & 6 & 6 & 6 & 6 & 5 & KKT  & 6 & 6 & 5 & 5 & 4 \\
kiwcresc & KKT  & 215 & 215 & 23 & 23 & 24 & KKT  & 152 & 154 & 22 & 22 & 22 \\
lakes & error & -- & -- & -- & -- & -- & error & -- & -- & -- & -- & -- \\
launch & error & -- & -- & -- & -- & -- & error & -- & -- & -- & -- & -- \\
lewispol & infeasible & 1 & 4 & 4 & 4 & 4 & infeasible & 1 & 4 & 4 & 4 & 4 \\
linspanh & KKT  & 2 & 2 & 2 & 2 & 1 & KKT  & 2 & 2 & 2 & 2 & 1 \\
loadbal & KKT  & 14 & 14 & 14 & 14 & 13 & KKT  & 14 & 14 & 14 & 14 & 13 \\
lootsma & KKT  & 16 & 16 & 16 & 16 & 15 & KKT  & 16 & 16 & 16 & 16 & 15 \\
lotschd & KKT  & 3 & 3 & 3 & 3 & 2 & KKT  & 3 & 3 & 3 & 3 & 2 \\
lsnnodoc & KKT  & 7 & 7 & 7 & 7 & 6 & KKT  & 7 & 7 & 7 & 7 & 6 \\
lsqfit & KKT  & 3 & 3 & 3 & 3 & 2 & KKT  & 3 & 3 & 3 & 3 & 2 \\
madsen & KKT  & 10 & 10 & 10 & 10 & 9 & KKT  & 10 & 10 & 10 & 10 & 9 \\
makela1 & KKT  & 18 & 18 & 9 & 9 & 8 & KKT  & 18 & 18 & 9 & 9 & 8 \\
makela2 & KKT  & 14 & 14 & 6 & 6 & 5 & KKT  & 14 & 14 & 6 & 6 & 5 \\
makela3 & KKT  & 83 & 83 & 28 & 28 & 27 & KKT  & 71 & 71 & 23 & 23 & 22 \\
makela4 & KKT  & 3 & 3 & 3 & 3 & 2 & KKT  & 3 & 3 & 3 & 3 & 2 \\
maratos & KKT  & 37 & 37 & 15 & 15 & 14 & KKT  & 32 & 32 & 14 & 14 & 13 \\
matrix2 & KKT  & 22 & 22 & 22 & 22 & 21 & KKT  & 22 & 22 & 22 & 22 & 21 \\
mconcon & KKT  & 5 & 5 & 5 & 5 & 4 & KKT  & 5 & 5 & 5 & 5 & 4 \\
mifflin1 & KKT  & 48 & 48 & 15 & 15 & 14 & KKT  & 39 & 39 & 13 & 13 & 12 \\
mifflin2 & KKT  & 56 & 56 & 14 & 14 & 13 & KKT  & 60 & 60 & 14 & 14 & 13 \\
minmaxbd & KKT  & 59 & 59 & 15 & 15 & 14 & KKT  & 74 & 74 & 19 & 19 & 18 \\
minmaxrb & KKT  & 4 & 4 & 4 & 4 & 3 & KKT  & 4 & 4 & 4 & 4 & 3 \\
mistake & KKT  & 7 & 7 & 7 & 7 & 6 & KKT  & 7 & 7 & 7 & 7 & 6 \\
model & KKT  & 3 & 3 & 3 & 3 & 2 & KKT  & 3 & 3 & 3 & 3 & 2 \\
mwright & KKT  & 23 & 23 & 22 & 22 & 21 & KKT  & 23 & 23 & 22 & 22 & 21 \\
nuffield & KKT  & 5 & 5 & 5 & 5 & 4 & KKT  & 5 & 5 & 5 & 5 & 4 \\
odfits & KKT  & 7 & 7 & 7 & 7 & 6 & KKT  & 7 & 7 & 7 & 7 & 6 \\
optcntrl & KKT  & 4 & 4 & 4 & 4 & 3 & KKT  & 4 & 4 & 4 & 4 & 3 \\
optmass & KKT  & 808 & 808 & 808 & 808 & 807 & KKT  & 808 & 808 & 808 & 808 & 807 \\
optprloc & KKT  & 8 & 8 & 6 & 6 & 5 & KKT  & 8 & 8 & 6 & 6 & 5 \\
orthregb & KKT  & 3 & 3 & 3 & 3 & 2 & KKT  & 3 & 3 & 3 & 3 & 2 \\
orthrege & error & -- & -- & -- & -- & -- & error & -- & -- & -- & -- & -- \\
pentagon & KKT  & 9 & 9 & 9 & 9 & 8 & KKT  & 9 & 9 & 9 & 9 & 8 \\
polak1 & KKT  & 93 & 93 & 15 & 15 & 15 & KKT  & 69 & 69 & 17 & 17 & 16 \\
polak2 & KKT  & 593 & 641 & 454 & 454 & 455 & KKT  & 553 & 587 & 451 & 451 & 451 \\
polak3 & KKT  & 53 & 53 & 15 & 15 & 14 & KKT  & 87 & 96 & 18 & 18 & 18 \\
polak4 & KKT  & 5 & 5 & 5 & 5 & 4 & KKT  & 5 & 5 & 5 & 5 & 4 \\
polak5 & KKT  & 30 & 30 & 25 & 25 & 24 & KKT  & 66 & 66 & 42 & 42 & 41 \\
polak6 & KKT  & 73 & 73 & 16 & 16 & 15 & KKT  & 67 & 67 & 15 & 15 & 14 \\
portfl1 & KKT  & 3 & 3 & 3 & 3 & 2 & KKT  & 3 & 3 & 3 & 3 & 2 \\
portfl2 & KKT  & 3 & 3 & 3 & 3 & 2 & KKT  & 3 & 3 & 3 & 3 & 2 \\
portfl3 & KKT  & 3 & 3 & 3 & 3 & 2 & KKT  & 3 & 3 & 3 & 3 & 2 \\
portfl4 & KKT  & 3 & 3 & 3 & 3 & 2 & KKT  & 3 & 3 & 3 & 3 & 2 \\
portfl6 & KKT  & 3 & 3 & 3 & 3 & 2 & KKT  & 3 & 3 & 3 & 3 & 2 \\
powellbs & KKT  & 15 & 15 & 15 & 15 & 14 & KKT  & 223 & 223 & 57 & 57 & 56 \\
powellsq & KKT  & 41 & 41 & 34 & 34 & 33 & error & -- & -- & -- & -- & -- \\
prodpl0 & KKT  & 9 & 9 & 9 & 9 & 8 & KKT  & 9 & 9 & 9 & 9 & 8 \\
prodpl1 & KKT  & 8 & 8 & 8 & 8 & 7 & KKT  & 8 & 8 & 8 & 8 & 7 \\
recipe & KKT  & 3 & 3 & 3 & 3 & 2 & KKT  & 3 & 3 & 3 & 3 & 2 \\
res & KKT  & 2 & 2 & 2 & 2 & 1 & error & -- & -- & -- & -- & -- \\
rk23 & KKT  & 10 & 10 & 10 & 10 & 9 & KKT  & 13 & 13 & 12 & 12 & 11 \\
robot & KKT  & 19 & 23 & 20 & 20 & 20 & KKT  & 21 & 25 & 21 & 21 & 21 \\
rosenmmx & KKT  & 75 & 75 & 14 & 14 & 13 & KKT  & 79 & 79 & 15 & 15 & 14 \\
s365mod & KKT  & 1333 & 1333 & 1333 & 1333 & 1332 & KKT  & 1339 & 1339 & 1337 & 1337 & 1336 \\
simpllpa & KKT  & 3 & 3 & 3 & 3 & 2 & KKT  & 3 & 3 & 3 & 3 & 2 \\
simpllpb & KKT  & 3 & 3 & 3 & 3 & 2 & KKT  & 3 & 3 & 3 & 3 & 2 \\
snake & KKT  & 3 & 3 & 3 & 3 & 2 & KKT  & 3 & 3 & 3 & 3 & 2 \\
spanhyd & KKT  & 6 & 6 & 6 & 6 & 5 & KKT  & 6 & 6 & 6 & 6 & 5 \\
spiral & KKT  & 215 & 215 & 107 & 107 & 106 & KKT  & 99 & 99 & 72 & 72 & 71 \\
ssnlbeam & error & -- & -- & -- & -- & -- & error & -- & -- & -- & -- & -- \\
stancmin & KKT  & 2 & 2 & 2 & 2 & 1 & KKT  & 2 & 2 & 2 & 2 & 1 \\
supersim & KKT  & 2 & 2 & 2 & 2 & 1 & KKT  & 2 & 2 & 2 & 2 & 1 \\
swopf & error & -- & -- & -- & -- & -- & error & -- & -- & -- & -- & -- \\
synthes1 & KKT  & 5 & 5 & 5 & 5 & 4 & KKT  & 5 & 5 & 5 & 5 & 4 \\
tame & KKT  & 2 & 2 & 2 & 2 & 1 & KKT  & 2 & 2 & 2 & 2 & 1 \\
try-b & KKT  & 9 & 9 & 9 & 9 & 8 & KKT  & 9 & 9 & 9 & 9 & 8 \\
twobars & KKT  & 8 & 8 & 8 & 8 & 7 & KKT  & 8 & 8 & 8 & 8 & 7 \\
vanderm4 & infeasible & 1 & 22 & 22 & 22 & 22 & infeasible & 1 & 22 & 22 & 22 & 22 \\
womflet & error & -- & -- & -- & -- & -- & error & -- & -- & -- & -- & -- \\
zangwil3 & KKT  & 2 & 2 & 2 & 2 & 1 & KKT  & 2 & 2 & 2 & 2 & 1 \\
zecevic2 & KKT  & 3 & 3 & 3 & 3 & 2 & KKT  & 3 & 3 & 3 & 3 & 2 \\
zecevic3 & KKT  & 9 & 9 & 8 & 8 & 7 & KKT  & 9 & 9 & 8 & 8 & 7 \\
zecevic4 & KKT  & 7 & 7 & 7 & 7 & 6 & KKT  & 7 & 7 & 7 & 7 & 6 \\
zigzag & KKT  & 131 & 131 & 131 & 131 & 130 & KKT  & 131 & 131 & 131 & 131 & 130 \\
zy2 & KKT  & 16 & 16 & 16 & 16 & 15 & KKT  & 16 & 16 & 16 & 16 & 15 \\
\hline
\end{longtable}
 
\end{document}